\documentclass{article}
\usepackage{graphicx} 
\usepackage{algorithm}
\usepackage{algorithmic}
\usepackage{amsmath,amsthm,amsfonts,amssymb,mathrsfs}
\usepackage{latexsym}
\usepackage{multicol}  
\usepackage{psfrag} 
\usepackage{multirow}

\theoremstyle{plain}
\newtheorem{theorem}{Theorem}[section]
\newtheorem{proposition}[theorem]{Proposition}
\newtheorem{lemma}[theorem]{Lemma}
\newtheorem{corollary}[theorem]{Corollary}

\theoremstyle{definition}
\newtheorem{definition}{Definition}[section]
\newtheorem{assumption}[definition]{Assumption}

\newtheorem{remark}[definition]{Remark}

\theoremstyle{remark}
\newtheorem{example}{Example}[section]

\numberwithin{equation}{section}
\makeatletter

\DeclareMathOperator\prox{prox}
\DeclareMathOperator\dom{dom}

\DeclareMathOperator\argmin{argmin}
\DeclareMathOperator\argmax{argmax}

\newcommand{\xh}{\hat{x}}

\DeclareMathOperator\trace{tr}

\DeclareMathOperator\sgn{sgn}

\newcommand{\beq}{\begin{equation}\nonumber} \newcommand{\eeq}{\end{equation}}

\newcommand{\bi}{\begin{itemize}}
\newcommand{\be}{\begin{enumerate}}
\newcommand{\ei}{\end{itemize}}
\newcommand{\ee}{\end{enumerate}}
\newcommand{\calH}{{\cal H}}

\newcommand{\R}{{\mathbb R}}
\newcommand{\N}{{\mathbb N}}

\newcommand{\calK}{{\cal K}}

\newcommand{\calO}{{\cal O}}

\newcommand{\calC}{{\cal C}}

\newcommand{\lb}{{\langle}}
\newcommand{\rb}{{\rangle}}

\def\boldf#1{\hbox{\rlap{$#1$}\kern.4pt{$#1$}}}

\newcommand{\trans}{^{\scriptscriptstyle \top}}

\newcommand{\rd}{\R^d}

\newcommand{\moreau}[2]{{#1}_{#2}}

\newcommand{\rnn}{\R^{n\times n}}

\begin{document}

\title{Hybrid Conditional Gradient - Smoothing Algorithms with Applications
to Sparse and Low Rank Regularization}

\author{A. Argyriou \\
{\' E}cole Centrale Paris, Center for Visual Computing \\
andreas.argyriou@ecp.fr \\
\\
M. Signoretto \\
KU Leuven, ESAT-STADIUS \\
Kasteelpark Arenberg 10, B-3001 Leuven, Belgium \\
marco.signoretto@esat.kuleuven.be \\
\\
J. Suykens \\
KU Leuven, ESAT-STADIUS \\
Kasteelpark Arenberg 10, B-3001 Leuven, Belgium \\
johan.suykens@esat.kuleuven.be
}

\date{}

\maketitle


\section{Introduction}
\label{sec:intro}

{\em Conditional gradient} methods are old and well studied
optimization algorithms. Their origin dates at least to the 50's and the Frank-Wolfe algorithm for 
quadratic programming \cite{frank_wolfe} but they apply to much more general optimization problems. 
General formulations of conditional gradient algorithms have been studied in the past
and various convergence properties of these algorithms have been proven.
Moreover, such algorithms have found application in many fields, such as optimal control, 
statistics, signal processing, computational geometry and machine learning. Currently, interest in 
conditional gradient methods is undergoing a revival because of their computational advantages
when applied to certain large scale optimization problems.
Such examples are {\em regularization} problems involving sparsity or low rank constraints,
which appear in many widely used methods in machine learning.

Inspired by such algorithms, in this chapter we study a first-order method 
for solving certain convex optimization problems. We focus on problems of the form
\beq
\min\left\{
f(x) + g(Ax) + \omega(x) : x\in \calH
\right\} \;.
\label{eq:intro}
\eeq
over a real Hilbert space $\calH$. We assume that $f$ is a convex function with {\em H{\"o}lder continuous gradient}, 
$g$ a {\em Lipschitz continuous} convex function, $A$ a bounded linear operator 
and $\omega$ a convex function defined over a {\em bounded domain}.
We also assume that the computational operations available are
the {\em gradient} of $f$, the {\em proximity operator} of $g$ and a {\em subgradient
of the convex conjugate} $\omega^*$.\footnote{For the precise assumptions 
required, see Assumptions \ref{ass:hybrid}, \ref{ass:oracle}.}
A particularly common type of problems covered by \eqref{eq:intro} is
\beq
\min\left\{
f(x) + g(Ax) : x\in \calC
\right\} \;,
\eeq
where $\calC$ is a bounded, closed, convex subset of $\calH$. Common such examples
are regularization problems with one or more penalties in the objective
(as the term $g\circ A$) and one penalty as a constraint described by $\calC$. 

Before presenting the algorithm, we review in Section \ref{sec:frank_wolfe} 
a generic conditional gradient algorithm which has been well studied in the past.
This standard algorithm 
can be used for solving problems of the form \eqref{eq:intro} whenever $g=0$.
However, the conditional gradient algorithm
cannot handle problems with a nonzero term $g$, because it would require computation of a subgradient of a composite
convex conjugate function, namely a subgradient of $(g\circ A + \omega)^*$.
In many cases of interest, there is no simple rule for such subgradients and 
the computation itself requires an iterative algorithm.

Thus, in Section \ref{sec:hybrid} we discuss an alternative approach that combines ideas 
from both conditional gradient algorithms and smoothing proximal algorithms, such as Nesterov smoothing. 
We call the resulting algorithm a {\em hybrid conditional gradient - smoothing} algorithm,
in short HCGS. This approach involves smoothing the $g$ term,
that is, approximating $g$ with a function whose gradient is Lipschitz continuous.
Besides this modification, HCGS is similar to the conditional gradient
algorithm. We show that, for suitable choices of the smoothing parameter, 
the estimates of the objective in HCGS converge to the minimum of \eqref{eq:intro}.
Moreover, the convergence rate is of the order of $\calO\left(\frac{1}{\varepsilon^2}\right)$  
iterations for attaining an accuracy of $\varepsilon$ in terms of the objective.
We do not claim originality, however, since similar theoretical results have appeared in 
recent work by Lan \cite{lan}. 

Our main focus is on highlighting applications of the hybrid approach on certain
applications of interest. 
To demonstrate the applicability of HCGS to regularization problems from 
machine learning, we present simulations on matrix problems with {\em simultaneous 
sparsity and low rank} penalizations. Examples of such applications are {\em graph denoising},
{\em link prediction} in social networks, {\em covariance estimation} and {\em sparse PCA}.
Each of these problems involves two penalties, an elementwise $\ell_1$ norm to promote
sparsity, and a {\em trace norm} to promote low rank. Standard algorithms 
may not be practical in high dimensional problems of this type.
As mentioned above, standard conditional gradient methods require a subgradient computation 
of a complicated function, whereas
proximal algorithms or subgradient based algorithms require 
an expensive singular value decomposition per iteration.
In contrast, HCGS requires only computation of dominant singular vectors,
which is more practical by means of the power iteration.
Thus, even though HCGS exhibits a slower asymptotic rate of
convergence than conditional gradient algorithms, Nesterov's method 
or the forward-backward algorithm, it scales much better to large 
matrices than these methods.


\section{Preliminaries from Convex Analysis}
\label{sec:convex}

Throughout the chapter, $\calH$ is a real Hilbert space endowed with norm $\|\cdot\|$
and inner product $\lb \cdot,\cdot\rb$.
As is standard in convex analysis, we consider extended value functions
$f:\calH \to (-\infty, +\infty]$ which can take the value $+\infty$.
With this notation, constraints can be written as indicator functions
that take the zero value inside the feasible set and $+\infty$
outside.

\begin{definition}
The domain of a function $f: \calH \to (-\infty, +\infty]$
is defined as the set $\dom f = \{x\in\calH : f(x) < +\infty \}$.
\end{definition}

\begin{definition}
The function $f: \calH \to (-\infty, +\infty]$ is called {\em proper}
if $\dom f \neq \varnothing$.
\end{definition}

\begin{definition}
The set of proper lower semicontinuous convex functions from $\calH$
to $(-\infty, +\infty]$ is denoted by $\Gamma_0(\calH)$.
\end{definition}


\begin{definition}
Let $f:\calH \to [-\infty, +\infty]$. The {\em convex conjugate} 
of $f$ is the function $f^*:\calH \to [-\infty, +\infty]$
defined as
\beq
f^*(x) = \sup\{ \lb u,x\rb - f(u) : u\in\calH\}
\eeq
for every $x\in\calH$.
\end{definition}

\begin{theorem}(Fenchel-Moreau) \cite[Thm. 13.32]{combettes_book}.
Every function $f\in \Gamma_0(\calH)$ is biconjugate,
\beq
f^{**}=f \,.
\eeq 
Moreover, $f^*\in\Gamma_0(\calH)$.
\label{thm:fenchel}
\end{theorem}

\begin{definition}
Let $f:\calH \to (-\infty, +\infty]$ be proper. A {\em subgradient} 
of $f$ at $x\in\calH$ is a vector $u\in\calH$ satisfying
\beq
\lb u, y-x\rb + f(x) \leq f(y)
\eeq
for every $y\in\calH$. The set of subgradients of $f$ at $x$ is called the
{\em subdifferential} of $f$ at $x$ and is denoted as $\partial f(x)$.
\end{definition}

\begin{proposition}
Let $f:\calH \to (-\infty,+\infty]$ be proper and $x\in\calH$. Then
$\partial f(x) \neq \varnothing$ implies $x\in \dom f$.
\label{prop:dom_grad}
\end{proposition}

\begin{theorem}
\cite[Thm. 16.23]{combettes_book}.
Let $f\in\Gamma_0(\calH), x\in\calH, u\in\calH$. Then 
\beq
u\in\partial f(x) \iff x \in \partial f^*(u) \iff f(x)+f^*(u) = \lb x,u\rb \;,
\eeq
\beq
(\partial f)^{-1} = \partial f^*.
\eeq
\label{thm:duality}
\end{theorem}

\begin{theorem}
\cite[Thm. 16.37]{combettes_book}.
Let $f \in \Gamma_0(\calH)$, $\calK$ a Hilbert space, $g \in \Gamma_0(\calK)$ 
and $A: \calH\to\calK$ a bounded linear operator. If $\dom g = \calK$ then 
$$\partial (f+g\circ A) = \partial f + A^*\circ\partial g \circ A\;.$$
\label{thm:grad_sum}
\end{theorem}

\begin{theorem}
\cite[Thm. 16.2]{combettes_book}.
Let $f:\calH \to (-\infty,+\infty]$ be proper. Then  
$$\argmin f = \{ x\in\calH : 0 \in \partial f (x) \} \;.$$
\label{thm:fermat}
\end{theorem}

\begin{definition}
The function $f\in\Gamma_0(\calH)$ is called $(p,L)$-{\em smooth}, 
where $L>0, p\in(0,1]$, if 
$f$ is Fr{\' e}chet differentiable on $\calH$ with a H{\" o}lder continuous gradient,
\begin{align*}
&&&& \|\nabla f(x) - \nabla f(y) \| \leq
L \, \|x-y\|^p  && \forall x,y \in \calH .
\end{align*}
\end{definition}

\noindent The case $p=1$ corresponds to functions with Lipschitz continuous gradient,
which appear frequently in optimization. 
The following lemma is sometimes called the {\em descent lemma}
\cite[Cor. 18.14]{combettes_book}.
\begin{lemma}
If the function $f\in\Gamma_0(\calH)$ is $(p,L)$-smooth then
\begin{align}
&& f(x) \leq f(y) + \lb x-y, \nabla f(y) \rb + \frac{L}{p+1} \|x-y\|^{p+1}
&& \forall x,y \in \calH.
\label{eq:descent}
\end{align}
\label{lem:descent}
\end{lemma}






\begin{theorem}
Let $\omega \in\Gamma_0(\calH)$. Then $\dom \omega$ is contained in
the ball of radius $\rho\in \R_+$, if and only if the convex conjugate
$\omega^*$ is $\rho$-Lipschitz continuous on $\calH$.
\label{thm:lip}
\end{theorem}

\begin{proof}
See \cite{rockafellar_level}, or \cite[Cor. 13.3.3]{rockafellar}
for a finite-dimensional version.
\end{proof}

\begin{corollary}
If the function $\omega \in\Gamma_0(\calH)$ has bounded domain then
$\partial \omega^*$ is nonempty everywhere on $\calH$.
\label{cor:lip}
\end{corollary}
\begin{proof}
Follows from Theorem \ref{thm:lip} and \cite[Prop. 16.17]{combettes_book}.
\end{proof}


\section{Generalized Conditional Gradient Algorithm}
\label{sec:frank_wolfe}

In this section, we review briefly the conditional gradient
algorithm in one of its many formulations. We focus on convex optimization
problems of a general type and discuss how a generalized conditional gradient
algorithm applies to such problems. We should note that this algorithm is
not the most generic formulation that has been studied -- see, for example, \cite{lan,demyanov2}-- 
but it covers a broad variety of optimization problems in machine learning.

Specifically, we consider the optimization problem
\begin{equation}
\min\left\{
f(x) + \omega(x) : x \in \calH
\right\}
\label{eq:opt}
\end{equation}
where we make the following assumptions.

\begin{assumption}
\label{ass:main}
\begin{itemize}
\item[]
\item $f,\omega \in\Gamma_0(\calH)$
\item $f$ is $(1,L)$-smooth
\item $\dom \omega$ is bounded, that is, there exists $\rho\in\R_{++}$ such that
$\|x\|\leq \rho,\forall x\in\dom \omega$.
\end{itemize}
\end{assumption}

\begin{remark}
Under Assumption \ref{ass:main}, the problem \eqref{eq:opt} admits a minimizer. The reason is that,
since  $\dom \omega$ is bounded, $\lim\limits_{\|x\|\to+\infty} \frac{f(x) + \omega(x)}{\|x\|} = +\infty$
(supercoercivity).
\label{rem:minimizer}
\end{remark}

\noindent The Fenchel dual problem 
associated with \eqref{eq:opt} is 
\begin{equation}
\max\left\{
-f^*(-z) - \omega^*(z) : z \in \calH
\right\} \;.
\label{eq:dual}
\end{equation}
Due to Fenchel's duality theorem and the fact that $\dom f = \calH$, the duality gap equals zero
and the maximum in \eqref{eq:dual} is attained \cite[Thm. 15.23]{combettes_book}.

\begin{algorithm}[t]
\caption{Generalized conditional gradient algorithm.}
\begin{algorithmic}
\STATE {\bf Input}~ $x_1 \in \dom \omega$
\FOR {$k=1,2,\dots$}
\STATE $\begin{aligned} & y_{k} \leftarrow ~\text{an element of}~\partial\omega^*\left(-\nabla f(x_k)\right) 
  && && && \text{(I)}  \\
 & x_{k+1} \leftarrow  (1-\alpha_k)x_k + \alpha_k y_k 
  && && && \text{(II)} \end{aligned}$
\ENDFOR 
\end{algorithmic}
\label{alg:fw}
\end{algorithm}

Algorithm \ref{alg:fw} has been used frequently in the past 
for solving problems of the type \eqref{eq:opt}.
It is a generalization of algorithms such
as the Frank-Wolfe algorithm for quadratic programming
\cite{frank_wolfe,gilbert} and conditional gradient algorithms \cite{kumar,demyanov,dunn}.
Algorithm \ref{alg:fw} applies to the general setting of convex
optimization problems of the form \eqref{eq:opt} which satisfy
Assumption \ref{ass:main}. In such general forms, the algorithm has been known and studied
for a long time in control theory and several of its convergence properties have been obtained \cite{kelley,kumar,dunn,demyanov,dunn2}.
More recently, interest in the family of conditional
gradient algorithms has been revived, especially in
theoretical computer science, machine learning, computational geometry
and elsewhere 
\cite{kale_hazan,hazan,jaggi,bach_fw,garber,tewari_greedy,ying_li,jaggi_arxiv,clarkson,harchaoui,lacoste,grigas}.
Some of these algorithms have appeared independently in various fields, such as statistics and signal processing, 
under different names and various guises. For example, it has been observed that conditional gradient methods
are related to boosting, greedy methods for sparse problems 
\cite{clarkson,greedy} and to orthogonal matching pursuit \cite{revisiting,jaggi_arxiv}.
Some very recent papers \cite{bach_fw,steinhardt} show an equivalence to the optimization method of 
mirror descent, which we discuss briefly in Section \ref{sec:connections}.

One reason for the popularity and the revival of interest in conditional gradient methods
has been their applicability to large scale problems. This advantage is evident, for example,
in comparison to proximal methods -- see \cite{combettes_pesquet} and references therein --
and especially in optimization problems involving matrices. 
Conditional gradient methods generally trade off a slower convergence rate (number of iterations)
for lower complexity of each iteration step. The accelerated 
proximal gradient methods \cite{nesterov07} benefit from the ``optimal'' 
$\calO\left(\sqrt{\frac{1}{\varepsilon}}\right)$ rate (where $\varepsilon$ is the accuracy
with respect to the optimization objective), whereas conditional gradient methods
exhibit a slower $\calO\left(\frac{1}{\varepsilon}\right)$ rate. 
On the other side, each step in the proximal methods requires computation of the
{\em proximity operator} \cite{moreau65,combettes_book} (see Section \ref{sec:hybrid}), 
which in some cases can be particularly costly. For example, the 
proximity operator of the {\em trace norm} of a matrix $X\in\R^{d\times n}$, 
\beq
\|X\|_{tr} = \sum_{i=1}^{\min\{d,n\}}\sigma_i(X) \;,
\eeq
where $\sigma_i(X)$ denote the singular values of $X$, 
requires computation of a complete singular value decomposition. In contrast, 
a conditional gradient method need only compute a dominant pair of left-right singular vectors, 
and such a computation scales better to large matrices \cite{jaggi}.

In general, as Algorithm \ref{alg:fw} indicates,
conditional gradient methods require computation of {\em dual subgradients}. Often,
this is a much less expensive operation than projection or the proximity operation. 
In other cases, proximity operations may not be feasible in a finite number of steps,
whereas dual subgradients are easy to compute. An obvious such case is $\ell_p$ or 
Schatten-$\ell_p$ regularization -- see \cite{revisiting}.
Other cases of interest occur when $\omega$ is the conjugate of a {\em max-function}.
Then the dual subgradient could be fast to compute while the proximity operation may be complex. 

Finally, another advantage of conditional gradient methods is that they build their estimate
of the solution incrementally. This implies that, in earlier iterations, time and space costs
will be low and that the algorithm may be stopped once an estimate of the desired parsimony is 
obtained (this could be, for example, a vector of certain sparsity or a matrix of certain rank).
Proximal methods, in contrast, do not necessarily obtain the desired parsimony until
later iterations (and even then it is not ``exact'').  

Formulation \eqref{eq:opt} covers many optimization problems studied so far
in the conditional gradients literature
and provides a concise description of variational problems
amenable to the standard conditional gradient algorithm. 
In Section \ref{sec:hybrid} we extend the applicability 
to problems with multiple penalties, 
by combining conditional gradients and smoothing techniques.

We remark that formulation \eqref{eq:opt} is valid in a generalized Hilbert space setting, 
so that it can be applied to infinite dimensional problems. This is particularly useful for {\em kernel methods}
in machine learning, for example, kernelized support vector machines or structured SVMs \cite{JST}
and nuclear or Schatten-$\ell_p$ regularization of operators \cite{abernethy}. 

To motivate Algorithm \ref{alg:fw}, consider the convex optimization problem \eqref{eq:opt}.
By Theorems \ref{thm:grad_sum} and \ref{thm:fermat}, 
$\xh\in \calH$ is a minimizer of \eqref{eq:opt} if and only if
\beq
0 \in \nabla f(\xh) + \partial \omega(\xh) 
\eeq
or, equivalently, 
\begin{equation}
- \nabla f(\xh) \in \partial \omega(\xh) \iff 
\xh \in \partial \omega^*(- \nabla f(\xh) ) \;,
\label{eq:fp}
\end{equation}
where we have used Theorem \ref{thm:duality}.
Thus, step (I) in Algorithm \ref{alg:fw} reflects the
fixed point equation \eqref{eq:fp}.
However, $\partial \omega^*(-\nabla f(x_k))$ is not a singleton in general 
and some elements of this set may be far from the minimizers of the problem. Hence step (II),
which weighs the new estimate with past ones, is necessary. 
With any affine weighting like that of step (II), the fixed point equation
\eqref{eq:fp} still holds.

\begin{remark}
Algorithm \ref{alg:fw} is {\em well defined},
since the subdifferential at step (I) is always nonempty, due to
Corollary \ref{cor:lip} and Assumption \ref{ass:main}.
\label{rem:well}
\end{remark}

Finally, let us note that several variants of Algorithm \ref{alg:fw}
are possible, in the spirit of the extensive literature on conditional
gradient methods. For example, there are various techniques (like line search) for
the choice of coefficients $\alpha_k$, more of the past iterates
may be used in (II) and so on.


\subsection{Convergence Rate}
\label{sec:rate}

\begin{theorem}
If, for every $k\in\N$, $\alpha_k \in [0,1]$, then
$x_k \in \dom \omega$ and
\begin{align*}
f(x_{k+1}) + \omega(x_{k+1}) - f(x) -\omega(x)  &\leq 
(1-\alpha_k) \bigl( f(x_k) + \omega(x_k) -  f(x) -\omega(x) \bigr) \\
& + 2\alpha_k^2 L \rho^2  
\end{align*}
for every $x\in\dom \omega$, $k\in\N$.
\label{thm:rate}
\end{theorem}

Theorem \ref{thm:rate} implies an $\mathcal{O}\left(\tfrac{1}{k}\right)$
convergence rate with respect to the objective values
$f(x_{k}) + \omega(x_{k}) - f(\xh) -\omega(\xh)$,
where $\xh$ is a minimizer of \eqref{eq:opt}.
This rate can be attained, for example, with the choice
$\alpha_k =\tfrac{2}{k+1}$. 

\begin{corollary}
If $\alpha_k = \frac{2}{k+1}$, for every $k\in\N$, then 
\begin{align}
f(x_{k+1}) + \omega(x_{k+1}) - f(x) -\omega(x)  &\leq 
\frac{8L\rho^2}{k+1}
\label{eq:rate1}
\end{align}
for every $x\in\dom \omega$, $k\in\N$.
\label{cor:rate1}
\end{corollary}

See \cite{demyanov,dunn,jaggi,clarkson,bach_fw} and references therein for these and related results,
as well as for bounds involving the duality gap estimates. 
It is also known that the lower bound for conditional gradient and similar algorithms 
is of the same order \cite{canon,revisiting,lan}.


\subsection{Connections to Mirror Descent and Gradient Descent}
\label{sec:connections}

It has been observed recently \cite{bach_fw,steinhardt} that the conditional gradient algorithm
is equivalent to a {\em mirror descent} algorithm in the dual.
The basic mirror descent algorithm \cite{mirror_descent,nemirovski,opt}
may be written as the iteration
\begin{equation}
x_{k+1} \leftarrow ~\text{an element of}~ x_k - t_k \partial \varphi(\nabla \psi^*(x_k)) \;,
\label{eq:mirror}
\end{equation}
where $t_k>0$ are step sizes, $\psi$ is strongly convex 
on a closed convex set $C$ and $\varphi$ is 
convex and Lipschitz continuous on $C$.
Setting $\omega = (\varphi\circ (-I))^*$, where $I$ denotes the identity operator, and $f=\psi^*$, algorithm
\eqref{eq:mirror} rewrites as a variant of Algorithm \ref{alg:fw}
(in which the update is not a convex combination). The set $C$ can
be viewed as the domain of $\omega$.

Consequently, when \eqref{eq:opt} is a proximity computation
(that is, when $f=\frac{1}{2\beta}\|\cdot\|^2$, $\beta>0$)
the conditional gradient algorithm \ref{alg:fw}
is equivalent to a {\em subgradient descent} in the dual.
In such cases $\nabla f = \frac{1}{\beta}I$ and Algorithm \ref{alg:fw}
becomes 
\beq
x_{k+1} \in (1-\alpha_k)x_k + \alpha_k \partial \omega^*\left(-\tfrac{1}{\beta}x_k\right) \;.
\eeq
Letting $h = \left(\omega^*\circ \left(-\tfrac{1}{\beta}I\right)\right)^*$, 
by the chain rule (Theorem \ref{thm:grad_sum}) this iteration is equivalent to 
\beq
x_{k+1} \in x_k - \alpha_k (I + \beta\, \partial h^*) (x_k) \;.
\eeq
In particular, when $\omega$ is $\mu$-strongly convex 
(and hence $\omega^*$ is (1,$\frac{1}{\mu}$)-smooth)
and $\alpha_k \leq \frac{\mu\beta}{1+\mu\beta}$ for every $k\in\N$, the above iteration
is equivalent to a {\em proximal point} algorithm
\cite{combettes_pesquet,eckstein,proximal_point}
because of \cite[Thm. 18.15]{combettes_book}.
Note that not all cases of subgradient descent are covered, since $\omega$ should have bounded domain,
implying that the dual objective function
should be a quadratic perturbation of a Lipschitz continuous function.


\section{Hybrid Conditional Gradient - Smoothing Algorithm}
\label{sec:hybrid}

We now introduce Algorithm \ref{alg:hybrid}, an extension of  
conditional gradient methods to optimization problems on bounded domains
which contain smooth and Lipschitz continuous terms.

\subsection{Description of the Hybrid Algorithm}
\label{sec:hybrid_algo}

Formally, we consider the class of optimization problems of the form
\begin{equation}
\min\left\{
f(x) + g(Ax) + \omega(x) : x\in \calH
\right\}
\label{eq:hybrid}
\end{equation}
where we make the following assumptions:

\begin{assumption}
\label{ass:hybrid}
\begin{itemize}
\item[]
\item $f,\omega \in\Gamma_0(\calH)$
\item $g \in\Gamma_0(\calK)$, $\calK$ is a Hilbert space
\item $A:\calH\to\calK$ is a bounded linear operator
\item $f$ is $(p,L_f)$-smooth
\item $g$ is $L_g$-Lipschitz continuous on $\calK$
\item $\dom \omega$ is bounded, that is, there exists $\rho\in\R_{++}$ such that
$\|x\|\leq \rho,\forall x\in\dom \omega$
\end{itemize}
\end{assumption}

\begin{remark}
Under Assumption \ref{ass:hybrid}, problem \eqref{eq:hybrid} admits a minimizer. 
As in Remark \ref{rem:minimizer}, the reason is growth of the objective function at infinity
(the objective equals $+\infty$ outside the feasible set, which is bounded).
\end{remark}

\noindent In order for the algorithm to be practical, we require that
\begin{assumption}
\label{ass:oracle}
\bi
\item[]
\item the gradient of $f$ is simple to compute at every $x\in\calH$,
\item a subgradient of $\omega^*$ is simple to compute at every $x\in\calH$,
\item the proximity operator of $\beta g$ is simple to compute
for every $\beta > 0,x\in\calH$.
\ei
\end{assumption}
\noindent The {\em proximity operator} was introduced by Moreau \cite{moreau65} as the (unique) minimizer 
\beq
\prox_{g}(x) = \argmin\left\{ \frac{1}{2}\|x-u\|^2 + g(u)
: u\in\calH\right\} \;.
\label{eq:prox}
\eeq 
For a review of the numerous applications of proximity operators to 
optimization, see, for example, \cite{combettes_wajs,combettes_pesquet}
and references therein. 

The following are some examples of optimization problems that belong to the
general class \eqref{eq:hybrid}. 
\begin{example} Regularization with two norm penalties:
\beq
\min \left \{ f(x) + \lambda \, \|x\|_a : \|x\|_b \leq B, x\in\R^d \right \} 
\label{eq:norms}
\eeq
where $f$ is $(p,L_f)$-smooth, $\lambda>0$ and $\|\cdot\|_a, \|\cdot\|_b$ can be any
norms on $\R^d$.
\label{ex:norms}
\end{example}

\begin{example} Regularization with a linear composite penalty and a norm:
\beq
\min \left \{ f(x) + \lambda \, \|Ax\|_a : \|x\|_b \leq B, x\in\R^d \right \} 
\label{eq:composite}
\eeq
where $f$ is $(p,L_f)$-smooth, $\lambda>0$, $\|\cdot\|_a, \|\cdot\|_b$ are
norms on $\R^\delta$, $\R^d$, respectively, and $A\in \R^{\delta\times d}$.
\label{ex:composite}
\end{example}

\begin{example} Regularization with multiple linear composite penalties and a norm:
\beq
\min \left \{ f(x) + \sum_{i=1}^n \lambda_i \, \|A_i x\|_{a_i} : \|x\|_b \leq B, x\in\R^d \right \} 
\label{eq:multicomposite}
\eeq
where $f$ is $(p,L_f)$-smooth and, for all $i\in\{1,\dots,n\}$, $\lambda_i>0$, $\|\cdot\|_{a_i}, \|\cdot\|_b$ are
norms on $\R^{\delta_i}$, $\R^d$, respectively, and $A_i\in \R^{{\delta_i}\times d}$.
Such problems can be seen as special cases of Example \ref{ex:composite} by applying the
classical direct sum technique, $\delta = \sum_{i=1}^n \delta_i$, $A = \left(\begin{smallmatrix}
A_1\\ \vdots \\ A_n \end{smallmatrix}\right)$, $\|(v_i)_{i=1}^n\|_a = \sum_{i=1}^n \lambda_i \, \|v_i\|_{a_i}$.
\label{ex:multicomposite}
\end{example}

\begin{algorithm}[t]
\caption{Hybrid conditional gradient - smoothing algorithm.}
\begin{algorithmic}
\STATE {\bf Input}~ $x_1 \in \dom \omega$
\FOR {$k=1,2,\dots$}
\STATE $\begin{aligned} & z_k \leftarrow - \nabla f(x_k) - \tfrac{1}{\beta_k} A^*A x_k 
+ \tfrac{1}{\beta_k} A^* \prox_{\beta_k g}(Ax_k)
  && && && \text{(I)}  \\
& y_{k} \leftarrow ~\text{an element of}~
\partial\omega^*(z_k) 
  && && && \text{(II)}  \\
 & x_{k+1} \leftarrow  (1-\alpha_k)x_k + \alpha_k y_k 
  && && && \text{(III)} \end{aligned}$
\ENDFOR 
\end{algorithmic}
\label{alg:hybrid}
\end{algorithm}

We propose to solve problems like the above with Algorithm \ref{alg:hybrid}.\footnote{
Algorithm \ref{alg:hybrid} is well-defined (see Remark \ref{rem:well}).}
We call it a {\em hybrid conditional gradient - smoothing} algorithm
(HCGS in short), because it involves a smoothing of function $g$ with parameter $\beta_k$. 
For any $\beta > 0$, the $\beta$-smoothing of $g$ is the Moreau envelope of $g$, that is, 
the function $g_\beta$ defined as
\begin{align}
&&&& g_\beta(x) := \min\left\{ \frac{1}{2\beta}\|x-u\|^2 + g(u) : u\in\calH\right\}
&& \forall x\in\calH \;.
\label{eq:moreau}
\end{align}
The function $g_\beta$ is a smooth approximation to $g$ (in
fact, the best possible approximation of $\frac{1}{\beta}$ smoothness), as summarized in
the following lemmas from the literature.

\begin{lemma}[Proposition 12.29 in \cite{combettes_book}]
\label{lem:moreau_grad}
Let $g\in\Gamma_0(\calH),\beta >0$. Then
$g_\beta$ is $(1,\frac{1}{\beta})$-smooth and its gradient can be obtained 
from the proximity operator of $g$ as: 
$\nabla g_\beta(x) = \frac{1}{\beta} \left(x - \prox_{\beta g}(x)\right)$.
\end{lemma}

\begin{lemma}
\label{lem:moreau_error}
Let $g\in\Gamma_0(\calH)$ be $L_g$-Lipschitz continuous and $\beta>0$. Then
\begin{itemize}
\item $g_\beta \leq g \leq g_\beta + \frac{1}{2}\beta L_g^2$
\item if $\beta \geq \beta'>0$, then  $g_\beta\leq g_{\beta'} \leq g_\beta + \frac{1}{2}(\beta-\beta')L_g^2$.
\end{itemize}
\end{lemma}

\begin{proof}
Define $\Psi_x(u)= \frac{1}{2\beta}\|x-u\|^2 + g(u)$.
We have $g_\beta(x) = \min_u \Psi_x(u) \leq \Psi_x(x) = g(x)$,
and this proves the left hand side of the first property.  For the
other side of the inequality, we have
\begin{align*}
\Psi_x(u) = \frac{1}{2\beta}\|x-u\|^2 + g(u) - g(x) + g(x)
\geq \frac{1}{2\beta}\|x-u\|^2 + g(x) - L_g \|x-u\|
\end{align*}
where we have used the Lipschitz property of $g$. This implies that
\begin{align*}
g_\beta(x) &= \min_u \Psi_x(u) \geq g(x) + \inf_u \left(\frac{1}{2\beta}\|x-u\|^2 - L_g \|x-u\| \right)
= g(x) - \frac{1}{2}\beta L_g^2~.
\end{align*}
The second property follows from the first one and
$\moreau{(\moreau{g}{\beta'})}{\beta-\beta'}=
\moreau{g}{\beta}$
(Proposition 12.22 in \cite{combettes_book}).
\end{proof}

\noindent Thus, the smoothing parameter $\beta$ controls the tradeoff
between the smoothness and the quality of approximation.  

At each iteration, the hybrid Algorithm \ref{alg:hybrid} computes the gradient 
of the {\em smoothed} part $f + g_{\beta_k}\circ A$, where $\beta_k$ is
the adaptive smoothing parameter. 
By Lemma \ref{lem:moreau_grad} and the chain rule, its gradient equals 
$$
\nabla (f + g_{\beta_k}\circ A)(x) = \nabla f(x) + \frac{1}{\beta_k} A^*\left(Ax - \prox_{\beta_k g}(Ax)\right) \;.
$$
The function $f$ is $(p,L_f)$-smooth and the function $g_{\beta_k}\circ A$ is $(1,\frac{1}{\beta_k}\|A\|^2)$-smooth.
By selecting $\beta_k$ that approaches $0$ as $k$ increases, we ensure that $g_{\beta_k}$ approaches $g$.

Algorithm \ref{alg:hybrid} can be viewed as an extension of conditional gradient algorithms 
and of Algorithm \ref{alg:fw}. But besides conditional gradient methods, the algorithm also exploits ideas
from proximal algorithms obtained by smoothing Lipschitz terms in the objective.
These methods are primarily due to Nesterov and have been successfully applied to many problems
\cite{nesterov_smooth,excessive,prisma,multicomposite,beck_smooth}. 
The smoothing we apply here is a type of Moreau envelope 
as in the variational problem \eqref{eq:moreau}, which is connected to Nesterov smoothing
-- see, for example, \cite{prisma,beck_smooth}.

However, unlike Nesterov's smoothing and other proximal methods, in our method we choose not to smooth
function $\omega$, or apply any other proximity-like operation to it. 
We do this because computation of the proximity operator of $\omega$ is not available in 
the settings which we consider here.\footnote{Note that Nesterov's smoothing would require $\omega$ 
to be Lipschitz continuous, and hence it does not apply directly to the case of bounded $\dom \omega$
but to a similar regularization problem with $\omega$ as a penalty in the objective.} 
For example, if $\omega$ expresses a trace norm constraint, the proximity
computation requires a {\em full} singular value decomposition, which does not scale well
with the size of the matrix. 
In contrast, the dual subgradient requires only computation
of a single pair of dominant singular vectors and this is feasible 
even for very large matrices using the power method or Lanczos algorithms.


\subsection{Convergence Rate}
\label{sec:hybrid_rate}

A bound on the convergence rate of the objective function can be obtained 
by first bounding the convergence of
the smoothed objective in a recursive way. The required number of iterations
is a function of $p$ and $\varepsilon$, where $p$ is the smoothing exponent of $f$ and
$\varepsilon$ is the accuracy in terms of the objective function. 
Regarding the proof technique, we should note that the HCGS Algorithm \ref{alg:hybrid} and the
proof of its convergence properties are mostly related to 
conditional gradient methods. On the other side, the proof technique does not 
share similarities with proximal methods such as ISTA or FISTA 
\cite{fista,nesterov83,nesterov_book}.

\begin{theorem}
Suppose that, for every $k\in\N$, $\alpha_k \in [0,1]$ and 
$\beta_k \geq \beta_{k+1} > 0$.
Let $F = f + g\circ A$, $F_k = f + g_{\beta_k}\circ A$. Then
$x_k \in \dom \omega$,
for every $k\in\N$, and 
\begin{multline}
F_{k+1}(x_{k+1}) + \omega(x_{k+1}) - F(x) -\omega(x)  \leq 
(1-\alpha_k) \bigl( F_k(x_k) + \omega(x_k) -  F(x) -\omega(x) \bigr) \\
+ \frac{(2\rho)^{p+1}L_f}{p+1}\alpha_k^{p+1} 
+ 2\|A\|^2 \rho^2 \frac{\alpha_k^2}{\beta_k}
    + \frac{1}{2}(\beta_k - \beta_{k+1})L_g^2 
\label{eq:hybrid_rate}
\end{multline}
for every $x\in\dom \omega$, $k\in\N$.
\label{thm:hybrid}
\end{theorem}

\begin{proof}
  For every $k\in \N$, we apply the descent lemma \ref{lem:descent}
  twice to obtain that
  \begin{align}\nonumber
    f(x_{k+1})  & \leq 
    f(x_k) + \lb \nabla f(x_k), x_{k+1}-x_k\rb + \frac{L_f}{p+1} \|x_{k+1}-x_k\|^{p+1}  \\
    & = f(x_k) + \alpha_k \lb \nabla f(x_k), y_k-x_k\rb + \frac{\alpha_k^{p+1}L_f}{p+1} \|y_k-x_k\|^{p+1} \;.
\label{eq:dl1}
  \end{align}
and
  \begin{align}\nonumber
    g_{\beta_k}(Ax_{k+1})  & \leq 
    g_{\beta_k}(Ax_k) + \lb \nabla (g_{\beta_k}\circ A)(x_k), x_{k+1}-x_k\rb + \frac{1}{2\beta_k}\|A\|^2 \|x_{k+1}-x_k\|^2  \\
    & = g_{\beta_k}(Ax_k) + \alpha_k \lb \nabla (g_{\beta_k}\circ A)(x_k), y_k-x_k\rb + \frac{\alpha_k^2}{2\beta_k}\|A\|^2 \|y_k-x_k\|^2 \;.
\label{eq:dl2}
  \end{align}
  Applying Lemma \ref{lem:moreau_error} to $g_{\beta_k}(Ax_{k+1})$ yields
  \begin{align*}
    g_{\beta_{k+1}}(Ax_{k+1})  \leq 
    g_{\beta_k}(Ax_k) + \alpha_k \lb \nabla (g_{\beta_k}\circ A)(x_k), y_k-x_k\rb & + \frac{\alpha_k^2}{2\beta_k}\|A\|^2 \|y_k-x_k\|^2 \\
    & + \frac{1}{2}(\beta_k - \beta_{k+1})L_g^2 \;.
  \end{align*}
Adding \eqref{eq:dl1} and \eqref{eq:dl2} we obtain that 
  \begin{align}\nonumber
    F_{k+1}(x_{k+1})  \leq 
    F_k(x_k) & + \alpha_k \lb \nabla F_k(x_k), y_k-x_k\rb + \frac{\alpha_k^{p+1}L_f}{p+1} \|y_k-x_k\|^{p+1} 
\\ & + \frac{\alpha_k^2}{2\beta_k}\|A\|^2 \|y_k-x_k\|^2 
    + \frac{1}{2}(\beta_k - \beta_{k+1})L_g^2 \;.
    \label{eq:proof_d}
  \end{align}

  By Theorem \ref{thm:duality} and Proposition \ref{prop:dom_grad},
  $y_k \in \dom \omega$ for every $k\in \N$.  Since $\alpha_k \in
  [0,1]$, applying an induction argument yields that $x_k\in \dom
  \omega$, for every $k\in \N$.  Thus, the values of the objective
  generated by the algorithm are finite.  From the construction of
  $y_k$ in steps (I), (II) and Theorem \ref{thm:duality}, we obtain that, for
  every $x \in \dom \omega$,
  \begin{align*}
    \lb y_k, -\nabla F_k(x_k) \rb - \omega(y_k) \geq 
    \lb x, -\nabla F_k(x_k) \rb - \omega(x) 
  \end{align*}
  and hence that 
  \begin{align*}
    \lb y_k-x_k, -\nabla F_k(x_k) \rb - \omega(y_k) & \geq 
    \lb x-x_k, -\nabla F_k(x_k) \rb - \omega(x) \\
    & \geq F_k(x_k) - F_k(x) -\omega(x) \;.
  \end{align*}
  Applying Lemma \ref{lem:moreau_error} to $g_{\beta_k}(Ax)$ yields
  \begin{align*}
    \lb y_k-x_k, -\nabla F_k(x_k) \rb - \omega(y_k) \geq 
    F_k(x_k) - F(x) - \omega(x) 
  \end{align*}
  and, therefore,
  \begin{align}
    \alpha_k \lb y_k-x_k, -\nabla F_k(x_k) \rb - \alpha_k \omega(y_k) \geq 
    \alpha_k F_k(x_k) - \alpha_k (F(x) +\omega(x) ) \;.
    \label{eq:proof_sub}
  \end{align}

  Adding \eqref{eq:proof_d} and \eqref{eq:proof_sub}, we obtain that
  \begin{multline*}
    F_{k+1}(x_{k+1}) + \alpha_k F_k(x_k) - \alpha_k (F(x) +\omega(x) ) \leq 
    F_k(x_k) - \alpha_k \omega(y_k) \\
+ \frac{\alpha_k^{p+1}L_f}{p+1} \|y_k-x_k\|^{p+1} 
+ \frac{\alpha_k^2}{2\beta_k}\|A\|^2 \|y_k-x_k\|^2 
    + \frac{1}{2}(\beta_k - \beta_{k+1})L_g^2 
  \end{multline*}
  or that
  \begin{multline*}
    F_{k+1}(x_{k+1}) + \omega(x_{k+1}) - F(x) -\omega(x) \leq
    (1-\alpha_k) \bigl( F_k(x_k) -  F(x) - \omega(x) \bigr) 
     + \omega(x_{k+1}) \\ - \alpha_k \omega(y_k) 
+ \frac{\alpha_k^{p+1}L_f}{p+1} \|y_k-x_k\|^{p+1} 
+ \frac{\alpha_k^2}{2\beta_k}\|A\|^2 \|y_k-x_k\|^2 
    + \frac{1}{2}(\beta_k - \beta_{k+1})L_g^2 \\
    \leq (1-\alpha_k) \bigl( F_k(x_k) + \omega(x_k) -  F(x) -\omega(x) \bigr) 
+ \frac{\alpha_k^{p+1}L_f}{p+1} \|y_k-x_k\|^{p+1} \\ 
+ \frac{\alpha_k^2}{2\beta_k}\|A\|^2 \|y_k-x_k\|^2 
    + \frac{1}{2}(\beta_k - \beta_{k+1})L_g^2 \;,
  \end{multline*}
  where the last step uses the convexity of $\omega$ and (III).
  Since $x_k,y_k \in \dom \omega$, for every $k\in\N$, it follows that 
  $\|y_k\|,\|x_k\|\leq \rho$ and hence that \eqref{eq:hybrid_rate} holds.
\end{proof}

\begin{corollary}
Suppose that, $\alpha_1=1$,  $\alpha_k \in [0,1]$ and 
$\beta_k \geq \beta_{k+1} > 0$, for every $k\in\N$.
Let $P_j = \prod_{i=j+1}^k (1-\alpha_i)$,  
for every $j\in\{1,\dots,k\}$.
Then
\begin{multline*}
f(x_{k+1}) + g(Ax_{k+1}) + \omega(x_{k+1}) - f(x) - g(Ax) - \omega(x) 
\leq 
 \frac{(2\rho)^{p+1}L_f}{p+1} \sum_{j=1}^k P_j\alpha_j^{p+1} \\
+ 2\|A\|^2 \rho^2 \sum_{j=1}^k P_j \frac{\alpha_j^2}{\beta_j} 
+ \frac{1}{2}L_g^2 \sum_{j=1}^k P_j (\beta_j - \beta_{j+1})
+ \frac{1}{2}\beta_{k+1} L_g^2 
\end{multline*}
for every $x\in\dom \omega$, $k\in\N$.
\label{cor:hybrid}
\end{corollary}

\begin{proof}
Let $D_k = f(x_k) + g_{\beta_k}(Ax_k) + \omega(x_k) - f(x) - g(Ax) - \omega(x)$.
Applying Theorem \ref{thm:hybrid}, we obtain 
\beq
D_{j+1} \leq (1-\alpha_j) D_j
+ \frac{(2\rho)^{p+1}L_f}{p+1}\alpha_j^{p+1} 
+ 2\|A\|^2 \rho^2 \frac{\alpha_j^2}{\beta_j}
    + \frac{1}{2}(\beta_j - \beta_{j+1})L_g^2 
\eeq
for every $j\in\{1,\dots,k\}$. Multiplying by $P_j$ and adding up, we obtain
\begin{multline*}
D_{k+1} \leq (1-\alpha_1) P_1 D_1
+ \frac{(2\rho)^{p+1}L_f}{p+1} \sum_{j=1}^k P_j\alpha_j^{p+1} 
+ 2\|A\|^2 \rho^2 \sum_{j=1}^k P_j \frac{\alpha_j^2}{\beta_j} \\
+ \frac{1}{2}L_g^2 \sum_{j=1}^k P_j (\beta_j - \beta_{j+1}) \\
= \frac{(2\rho)^{p+1}L_f}{p+1} \sum_{j=1}^k P_j\alpha_j^{p+1} 
+ 2\|A\|^2 \rho^2 \sum_{j=1}^k P_j \frac{\alpha_j^2}{\beta_j} 
+ \frac{1}{2}L_g^2 \sum_{j=1}^k P_j (\beta_j - \beta_{j+1}) 
\end{multline*}
Applying Lemma \ref{lem:moreau_error} to $g_{\beta_{k+1}}(Ax_{k+1})$ the assertion follows. 
\end{proof}

\begin{corollary}
If $\alpha_k = \frac{2}{k+1}$, $\beta>0$ and 
$\beta_k = \frac{\beta}{\sqrt{k}}$, for every $k\in\N$, then
\begin{multline*}
f(x_{k+1}) + g(Ax_{k+1}) + \omega(x_{k+1}) - f(x) - g(Ax) - \omega(x) 
\leq \\
\frac{(4\rho)^{p+1} L_f}{(p+1)(k+1)^p} 
+ \frac{8\rho^2\|A\|^2}{\beta\sqrt{k+1}}
+ \frac{1}{2}L_g^2 \beta\frac{\sqrt{k+2}}{k}
+ \frac{L_g^2\beta}{2\sqrt{k+1}}
\end{multline*}
for every $x\in\dom \omega$, $k\in\N$.
\label{cor:hybrid3}
\end{corollary}

\begin{proof}
It follows easily from Corollary \ref{cor:hybrid} and the computation $P_j = \frac{j(j+1)}{k(k+1)}$.
\end{proof}

We notice that when $p\geq \frac{1}{2}$ the asymptotic rate does not depend on $p$ and translates to
$\calO\left(\frac{1}{\varepsilon^2}\right)$ iterations, if $\varepsilon$ is the precision in terms of
the objective function. This rate of convergence is an 
order of magnitude slower than the rate for the standard conditional gradient algorithm
(Corollary \ref{cor:rate1}). Thus, the extended flexibility of handling multiple 
additional penalties (function $g$) and the Moreau smoothing incur a cost in terms of iterations.
In other words, the class of optimization problems to which the hybrid algorithm applies
is significantly larger than that of the standard algorithm \ref{alg:fw} and 
a deterioration in the rate of convergence is inevitable.
When $0<p<\frac{1}{2}$, the bound is dominated by the term involving $p$ and the number of
iterations required grows as $\calO\left(\varepsilon^{-\frac{1}{p}}\right)$. 

If there is no $g\circ A$ term ($A=0, g=0$) then the algorithm becomes 
the standard conditional gradient and the corollary reduces to known bounds for
standard conditional gradient methods. The number of iterations grows as
$\calO\left(\varepsilon^{-\frac{1}{p}}\right)$, which ranges from 
$\calO\left(\frac{1}{\varepsilon}\right)$ (for $p=1$) to 
impractical when $f$ is too ``close'' to a Lipschitz continuous function ($p\simeq 0$).

The rate in Corollary \ref{cor:hybrid3} is also slower 
than the $\calO\left(\frac{1}{\varepsilon}\right)$ rates obtained with smoothing methods, such as
\cite{prisma,nesterov_smooth,multicomposite,beck_smooth}. However, smoothing methods require 
a more powerful computational oracle (the proximity operator of $\omega$ instead of the dual subgradient)
and hence may be inapplicable in problems like those involving very large matrices, 
because computation of $\prox_\omega$ may not scale well. Another $\calO\left(\frac{1}{\varepsilon^2}\right)$ 
alternative is subgradient methods, but these may be inapplicable too for similar reasons. 
For example, the subgradient of the trace norm as either a penalty 
term or a constraint requires a full singular value decomposition.

In addition, like other conditional gradient methods or greedy methods and matching pursuits,
the HCGS algorithm \ref{alg:hybrid} builds
a parsimonious solution in additive fashion rather than starting from a complex solution and 
then simplifying it.
This feature may be desirable in itself whenever a parsimonious solution is sought. 
For example, in many cases it is more important to obtain a
sparse or low rank estimate of the solution rather than a more accurate one with many small 
nonzero components or singular values. In machine learning problems,
especially, this is frequently the case since the optimization objective is just an approximation
of the ideal measure of expected risk \cite{greedy}.  
Another advantage of such algorithmic schemes is computational. In sparse estimation problems regularized with
an $\ell_1$ constraint, the data matrix or the dictionary may be huge and hence computation of 
$\nabla f(x)$ may be feasible only for sparse vectors $x$ (when $f$ is a quadratic function). 
Moreover, such a computation can be done efficiently since the gradient from the previous
iteration can be reused, due to update (III).

\subsection{Minimization of Lipschitz Continuous Functions}
\label{sec:lipschitz}

A special case of particular interest occurs when $f=0$, that is,
when there is no smooth part. Then HCGS solves the optimization problem
\begin{equation}
\min\left\{
g(Ax) + \omega(x) : x\in \calH
\right\}
\label{eq:lip}
\end{equation}
under Assumption \ref{ass:hybrid} as before. Namely, the objective
function consists of a Lipschitz term $g$ and a generic term $\omega$
defined on a bounded domain. 
For example, such a problem is the minimization of a Lipschitz continuous
function over a bounded domain. 
More generally, $g\circ A$ may incorporate a sum of multiple Lipshitz continuous 
penalties.

The HCGS algorithm specified to problem \eqref{eq:lip} is the same as
Algorithm \ref{alg:hybrid} with $\nabla f(x_k)$ removed. In this way, the
computational model of conditional gradient methods extends from
minimization of smooth functions to minimization of Lipschitz
continuous functions. Moreover, the convergence rate deteriorates from
$\calO\left(\frac{1}{\varepsilon}\right)$ for smooth functions to
$\calO\left(\frac{1}{\varepsilon^2}\right)$ for Lipschitz functions --
which is not surprising, since the latter are in general more
difficult to optimize than the former. This fact has been shown recently by Lan for 
several conditional gradient algorithms \cite{lan}. Lan has also
shown that these rates coincide with the lower complexity bounds for 
a family of algorithms involving $\partial \omega^*$ oracles. 
The above fact is also intriguing in view of the analogy to the results 
known about Nesterov's proximal methods \cite{nesterov_book}. 
Those methods, under a more powerful computational oracle for $\omega$, 
exhibit an $\calO\left(\frac{1}{\sqrt{\varepsilon}}\right)$ 
rate when $g$ is smooth versus an $\calO\left(\frac{1}{\varepsilon}\right)$
rate when $g$ is Lipschitz continuous.

\subsection{Implementation Details}
\label{sec:implementation}

It is worth noting that the HCGS algorithm does not require knowledge of the 
Lipschitz constants $L_f, L_g$ and can be implemented with an arbitrary choice of $\beta$.
An alternative is to optimize the bound in Corollary \ref{cor:hybrid3} with respect
to $\beta$, which gives an optimal choice of
$\frac{2\sqrt{2}\rho\|A\|}{L_g}$, asymptotically. If the desired accuracy
$\varepsilon$ can be specified in advance, then the optimal $\beta$
will also depend on $\varepsilon$. Computing such a $\beta$ value
is possible only if the Lipschitz constant and bound of the
optimization problem are available, but for regularization problems
these constants can be computed from the regularization parameters. 

For $p\geq\frac{1}{2}$, these two constants, $\rho$ and $L_g$, have the largest influence in
the convergence rate, since they appear in the $\calO\left(\frac{1}{\sqrt{k}}\right)$ 
terms that dominate the bound. The constant $\rho$ cannot be changed, since 
it is a property of the feasibility domain. However $L_g$ can be reduced
by rescaling the objective function and hence it can become independent of the dimensionality
of the problem.

Some care may be needed to tackle numerical issues arising from
very small values of $\beta_k$ as $k$ becomes large. These issues affect only
step (I), whereas the computation of $y_k$ in step (II) remains always inside the $\rho$-ball,
since $\omega^*$ is $\rho$-Lipschitz continuous. Moreover, for large $k$, the past estimates
dominate the update (III) and hence the effect of any numerical issues diminishes as $k$ grows.

\section{Applications}
\label{sec:special}

We now instantiate the HCGS algorithm \ref{alg:hybrid} to some special cases
which appear in applications and we present the corresponding algorithms.
These examples are only a sample and do not cover the whole range of possible applications.
First, consider the problem of learning a {\em sparse and low rank matrix} by 
regularization with the $\ell_1$ norm and a trace norm constraint \cite{emile},
\begin{equation}
\min \{ f(X) + \lambda \|X\|_1 : \|X\|_{tr} \leq B , X \in \R^{n\times n} \} \;,
\label{eq:sparse_low_rank}
\end{equation}
where $\|\cdot\|_1$ denotes the elementwise $\ell_1$ norm of a matrix and
$\|\cdot\|_{tr}$ the trace norm (or nuclear norm). The strongly smooth function $f$
expresses an error term (where the dependence on the data is absorbed in $f$)
and may arise by using, for example, the square loss or the logistic loss. 
This setting has been proposed for applications such as graph denoising or
prediction of links on a social network.
The resulting algorithm (Algorithm \ref{alg:sparse_low_rank})
depends on the proximity operator of the $\ell_1$
norm, also known as the soft thresholding operator,
\beq
\mathcal{S}(X;\gamma) = \sgn(X)\odot (|X|-\gamma)_+ \;,
\label{eq:soft}
\eeq  
where $\sgn, \odot, |\cdot|$ denote elementwise sign, multiplication and absolute value
on matrices and $(\cdot)_+$ the positive part elementwise.

Note that the same algorithm can be used for solving a variation of 
\eqref{eq:sparse_low_rank} that restricts the optimization to 
the space of symmetric matrices. This may occur, for example, when learning
the adjacency matrix of an undirected graph. 
One should ensure, however, that the initial matrix $X_1$ 
is symmetric.

\begin{algorithm}[t]
\caption{Hybrid algorithm for sparse - low rank problems.}
\begin{algorithmic}
\STATE {\bf Input}~ $X_1\in\rnn$ such that $\|X_1\|_{tr} \leq B$
\FOR {$k=1,2,\dots$}
\STATE $Z_k \leftarrow - \nabla f(X_k) - \tfrac{1}{\beta_k} X_k 
+ \tfrac{1}{\beta_k} \mathcal{S}(X_k;\beta_k\lambda)$\vspace{0.5ex}
\STATE $(u_k,v_k)\leftarrow$~ a left and right pair of singular vectors of $Z_k$
corresponding to the largest singular value \vspace{0.5ex}
\STATE $Y_k \leftarrow B \, u_k v_k\trans$ \vspace{0.5ex}
\STATE $X_{k+1} \leftarrow  (1-\alpha_k)X_k + \alpha_k Y_k$ 
\ENDFOR 
\end{algorithmic}
\label{alg:sparse_low_rank}
\end{algorithm}

A problem which shares some similarities with the previous one 
is the convex relaxation of {\em sparse PCA} proposed in \cite{aspremont},
\begin{equation}
\max \{ \lb C, X\rb -\lambda \|X\|_1 : \trace(X) = 1, X \succeq 0, X \in \R^{n\times n} \} \;.
\label{eq:sparse_pca}
\end{equation} 
Solving this optimization problem can be used for finding a dominant sparse eigenvector
of $C$, which is a prescribed $n\times n$ symmetric matrix.
The problem falls under the framework \eqref{eq:hybrid} with $f$ being a linear
function and $\omega$ the indicator function of the (bounded) spectrahedron 
$\{X \in \R^{n\times n}: \trace(X) = 1, X \succeq 0 \}$.
Computation of a dual subgradient amounts to computing a
solution of the problem
\beq
\max\{ \lb Y, Z\rb : \trace(Y) = 1, Y \succeq 0, Y \in \R^{n\times n} \} 
\label{eq:spectra}
\eeq
for a given symmetric matrix $Z\in\R^{n\times n}$. It is easy to see that 
this computation requires a dominant eigenvector of $Z$. This results
in Algorithm \ref{alg:sparse_pca}.

\begin{algorithm}[t]
\caption{Hybrid algorithm for sparse PCA relaxation.}
\begin{algorithmic}
\STATE {\bf Input}~ $X_1\in\rnn$ such that $\trace(X_1) = 1, X_1 \succeq 0$
\FOR {$k=1,2,\dots$}
\STATE $Z_k \leftarrow C - \tfrac{1}{\beta_k} X_k 
+ \tfrac{1}{\beta_k} \mathcal{S}(X_k;\beta_k\lambda)$\vspace{0.5ex}
\STATE $u_k\leftarrow$~ a dominant eigenvector of $Z_k$ \vspace{0.5ex}
\STATE $Y_k \leftarrow u_k u_k\trans$ \vspace{0.5ex}
\STATE $X_{k+1} \leftarrow  (1-\alpha_k)X_k + \alpha_k Y_k$ 
\ENDFOR 
\end{algorithmic}
\label{alg:sparse_pca}
\end{algorithm}

A related problem is to restrict the sparse - low rank optimization 
\eqref{eq:sparse_low_rank} to the cone of positive semidefinite matrices.
This problem has been proposed for estimating a covariance matrix in \cite{emile}.
Since the trace norm is equal to the trace on the positive semidefinite cone,
the algorithm is similar to Algorithm \ref{alg:sparse_pca}. The only differences are
the initialization, a general smooth function $f$ and a factor of $B$ in the
update of $Y_k$.

A third example of an optimization problem that falls under our framework
is a regularization problem with $\ell_1$ and additional penalties,
\begin{equation}
\min \left\{ \frac{1}{2}\lb x, Q x\rb + \lb c, x\rb + 
g(Ax) : \|x\|_1 \leq B, x\in\rd \right \} \;.
\label{eq:qp}
\end{equation}
Here $Q\in\R^{d\times d}$ is a prescribed positive semidefinite matrix, 
$c\in\rd$ a prescribed vector and $g,A$ satisfy Assumption \ref{ass:hybrid}.
For example, $\eqref{eq:qp}$ could arise from an estimation or learning
problem, the quadratic part corresponding to the data fit term.
The $\ell_1$ constraint is used to favor sparse solutions.
The penalty terms $g\circ A$ may involve multiple norms whose
proximity operator is simple to compute, such as the 
group Lasso norm \cite{group_lasso}, total variation norms \cite{rof} etc.

The hybrid method, specialized to such problems, is shown in Algorithm \ref{alg:l1}.
In general, several other algorithms may be used for solving problems like \eqref{eq:qp}
(smoothing, Douglas-Rachford, subgradient methods etc.), but here we are interested
in cases with {\em very large dimensionality} $d$. 
In such cases, computation of the gradient at an arbitrary vector is $\calO(d^2)$ and very costly.
On the other side, in the HCGS algorithm, $x_{k+1}$ is $(k+1)$-sparse
and computing the new gradient $Qx_{k+1}$ can be done efficiently by keeping $Qx_k$ in memory and 
computing $Qy_k$, which is proportional to the $j$-th column of $Q$. 
The latter requires only $\calO(d)$ operations, or $\calO(dm)$
if $Q$ is the square of an $m\times d$ data matrix. 
Thus, HCGS can be applied to such problems at a smaller cost,
by starting with an initial cardinality-one vector and 
stopping before $k$ becomes too large.

There is also an interesting interpretation of Algorithm \ref{alg:l1} as an extension
of {\em matching pursuits} \cite{mallat,tropp} to problems with multiple penalties. 
Assuming that $Q$ is the square of a matrix of dictionary elements (or more generally
a Gram matrix of elements from a Hilbert space), then the algorithm shares similarities with 
orthogonal matching pursuit (OMP). Indeed, such a connection has already been observed
for the standard conjugate gradient (which corresponds to absence of the $g\circ A$ term)
\cite{revisiting,jaggi_arxiv}, the main difference from OMP being in the update of $x_{k+1}$.
Similarly, the HCGS algorithm \ref{alg:l1} could be phrased as an extension of OMP that imposes 
additional penalties $g\circ A$, besides sparsity, on the coefficients of the atoms.
For example, $g\circ A$ could involve {\em structured sparsity} penalties
(such as penalties for group, hierarchical or graph sparsity) and then HCGS would yield
a scalable alternative to structured variants of OMP \cite{huang} or proximal methods
for structured sparsity \cite{mairal}.

\begin{algorithm}[t]
\caption{Hybrid algorithm for sparse multicomposite problems (\ref{eq:qp}).}
\begin{algorithmic}
\STATE {\bf Input}~ $x_1 = Be_i$ for some $i\in\{1,\dots,d\}$
\FOR {$k=1,2,\dots$}
\STATE $z_k \leftarrow - Qx_k-c - \tfrac{1}{\beta_k} A^*A x_k 
+ \tfrac{1}{\beta_k} A^* \prox_{\beta_k g}(Ax_k)$ \vspace{0.5ex}
\STATE $y_{k} \leftarrow B\, \sgn((z_k)_j) e_j$, 
where $j\in\argmax\nolimits_{i=1}^d |(z_k)_j|$ \vspace{0.5ex}
\STATE $x_{k+1} \leftarrow  (1-\alpha_k)x_k + \alpha_k y_k$ \vspace{0.5ex}
\ENDFOR 
\end{algorithmic}
\label{alg:l1}
\end{algorithm}







\section{Simulations}
\label{sec:sim}

\subsection{Simultaneous Sparse and Low Rank Regularization}

In this section we focus on testing Algorithm
\ref{alg:sparse_low_rank} (HCGS) on the estimation of simultaneously
sparse and low rank matrices.\footnote{Code is available at {\tt http://cvn.ecp.fr/personnel/andreas/code/index.html}}
Our aim is to compare the procedure with
the proximal algorithms proposed, for the same task, in
\cite{emile}. The experiments illustrate the fact that
HCGS scales better than the SVD-based alternatives.

 \begin{figure}   
\hspace{-1.5cm}
\begin{tabular}{cc}
	\begin{minipage}{7cm}
  		\includegraphics[width=1\textwidth ]{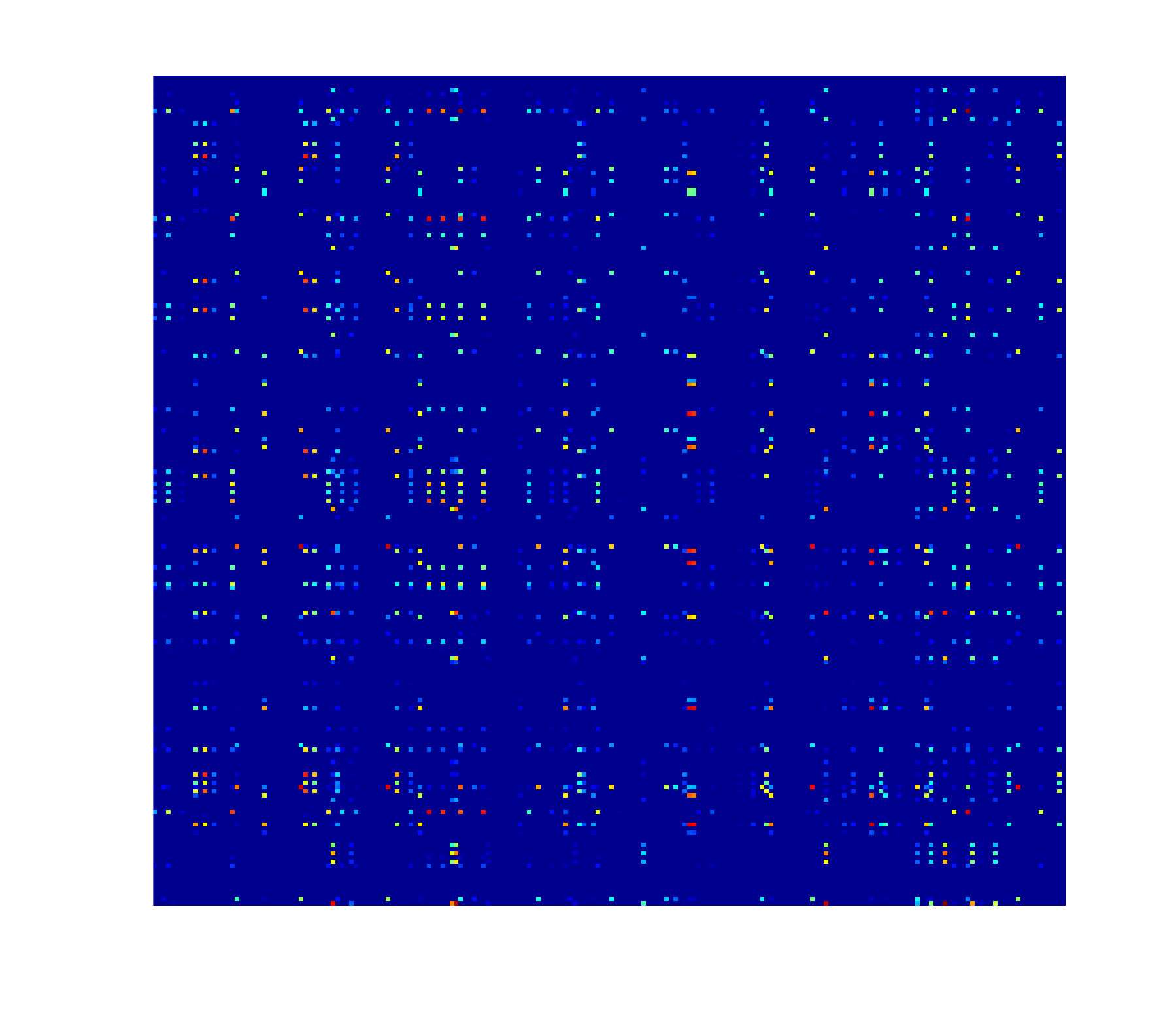}
\begin{center}
	\vspace{-.5cm}
	(a)   
\end{center}
\end{minipage}		
   & 
	\begin{minipage}{7cm}
  		\includegraphics[width=1\textwidth ]{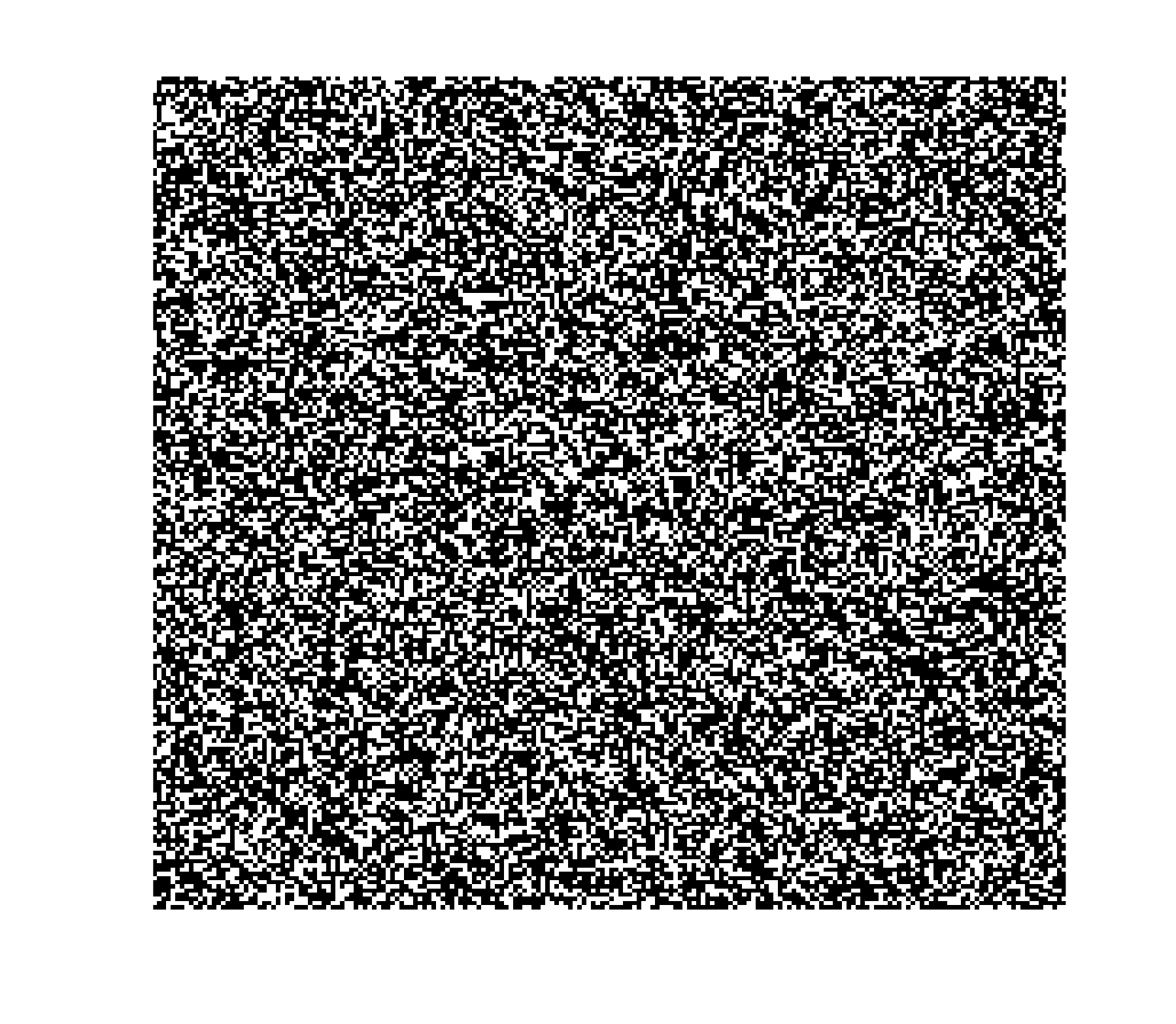}
\begin{center}
	\vspace{-.5cm}
	(b)   
\end{center}
\end{minipage}\\
	\begin{minipage}{7cm}
  		\includegraphics[width=1\textwidth ]{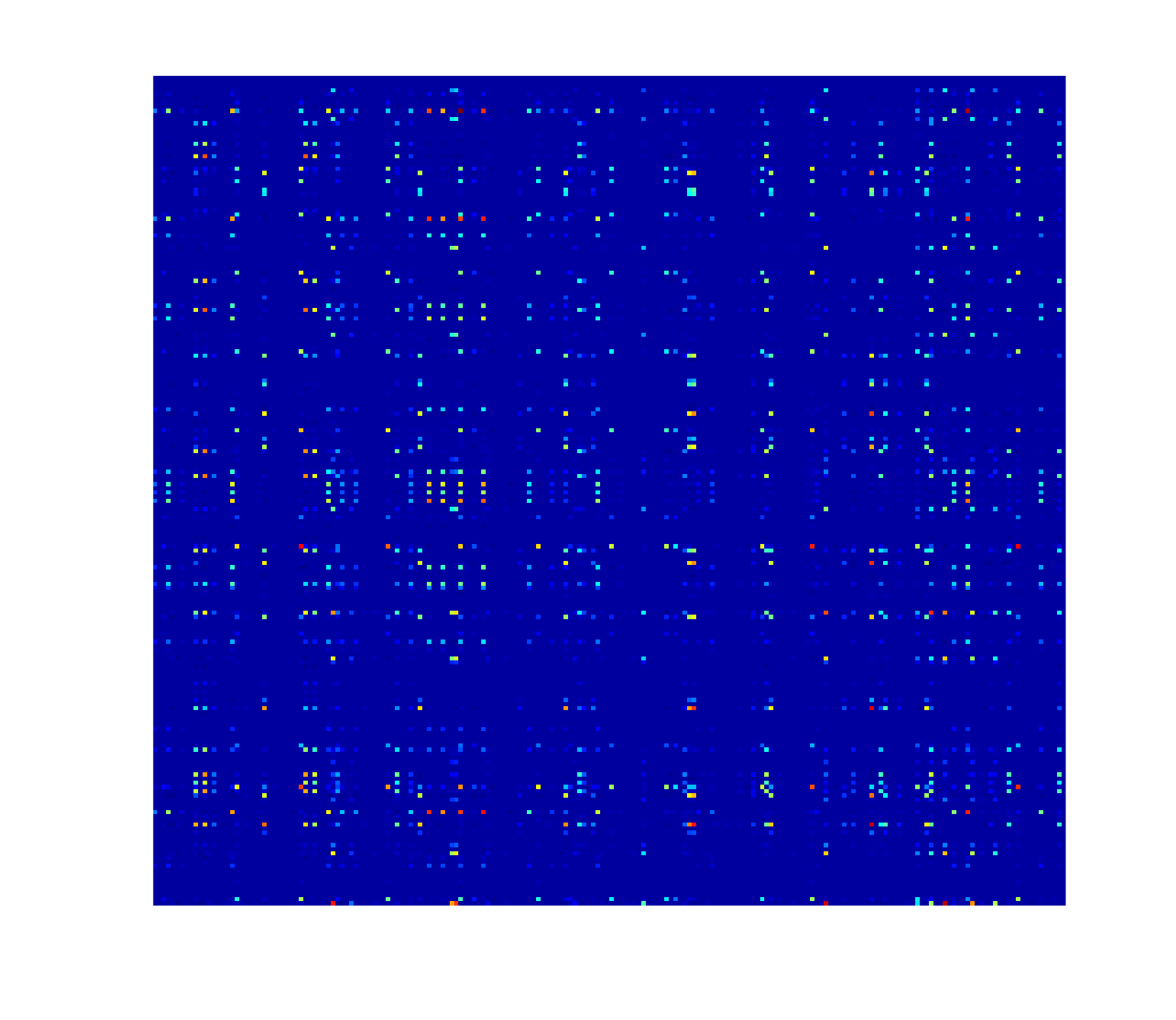}
\begin{center}
	\vspace{-.5cm}
(c)   
\end{center}
\end{minipage}		
   & 
	\begin{minipage}{7cm}
  		\includegraphics[width=1\textwidth ]{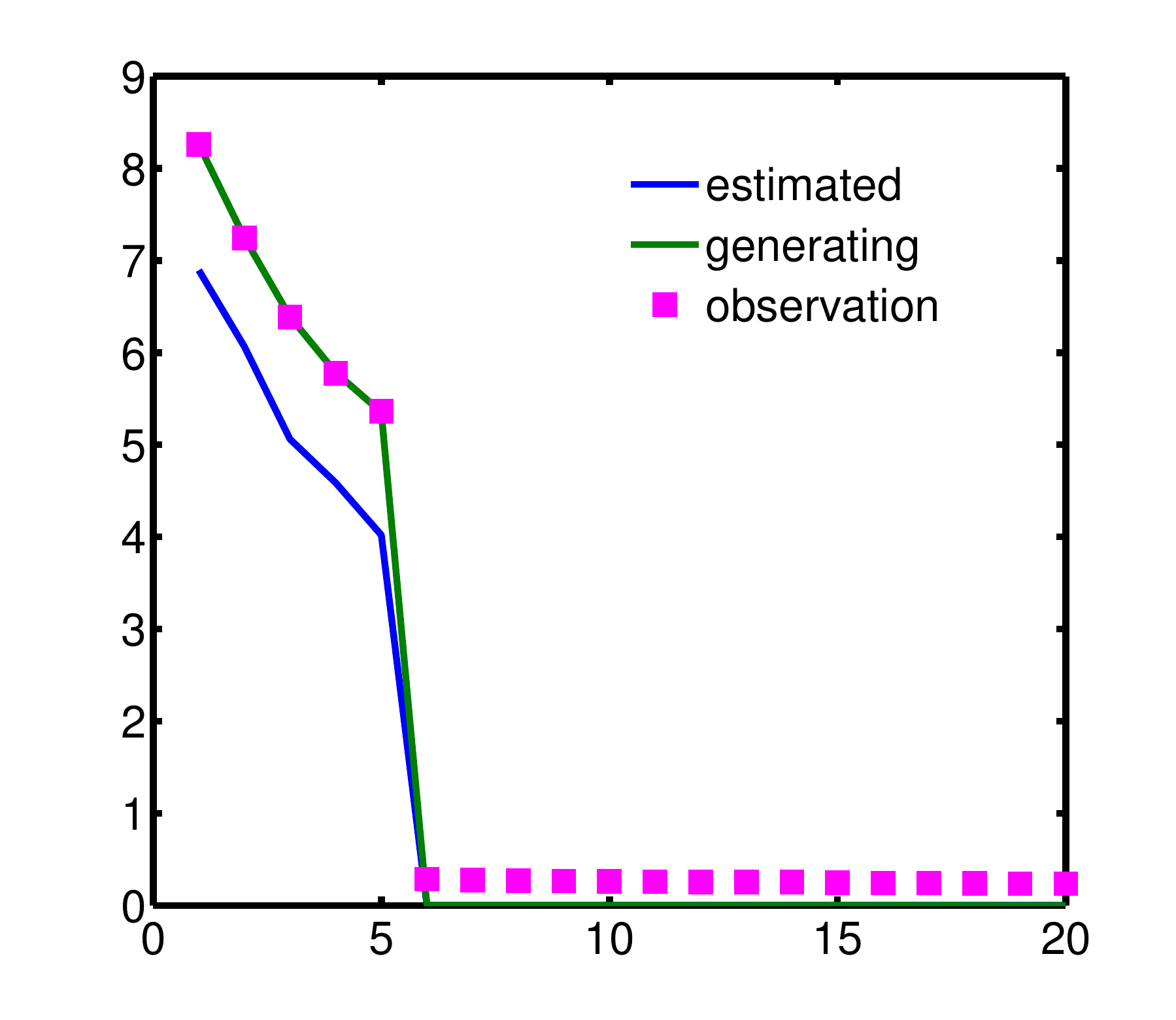}
\begin{center}
	\vspace{-.5cm}
	(d)   
\end{center}
\end{minipage}
\end{tabular} 
 \caption{An illustration of the synthetic problem for
      $N=200$. The generating matrix, simultaneously sparse and
      low-rank (a), a mask with the observed entries, in
      white; $40\%$ of the total number of entries are observed
      (b), the matrix estimated by HCGS (c),
      the leading singular values for the
      generating/observation/estimated matrix
      (d).\label{fig:compfig}} 
  \end{figure}

\begin{figure} 
	\hspace{-1.5cm}
	\begin{tabular}{cc}
			\begin{minipage}{7cm}
		  		\includegraphics[width=1\textwidth ]{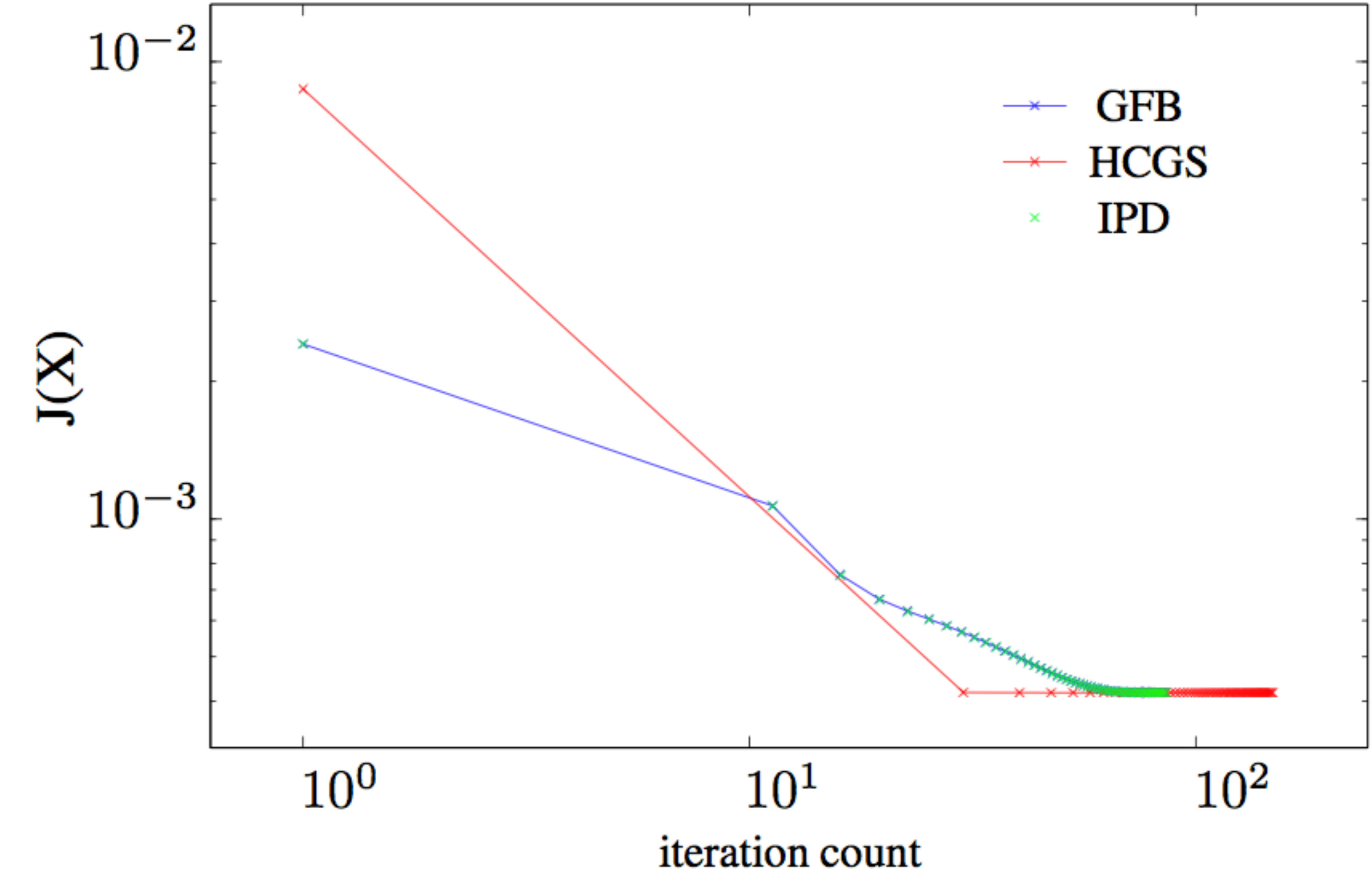}
		\begin{center}
			\vspace{-.5cm}
		~~~~~~~~~~	(a)   
		\end{center}
		\end{minipage}		
		   & 
			\begin{minipage}{7cm}
		  		\includegraphics[width=1\textwidth ]{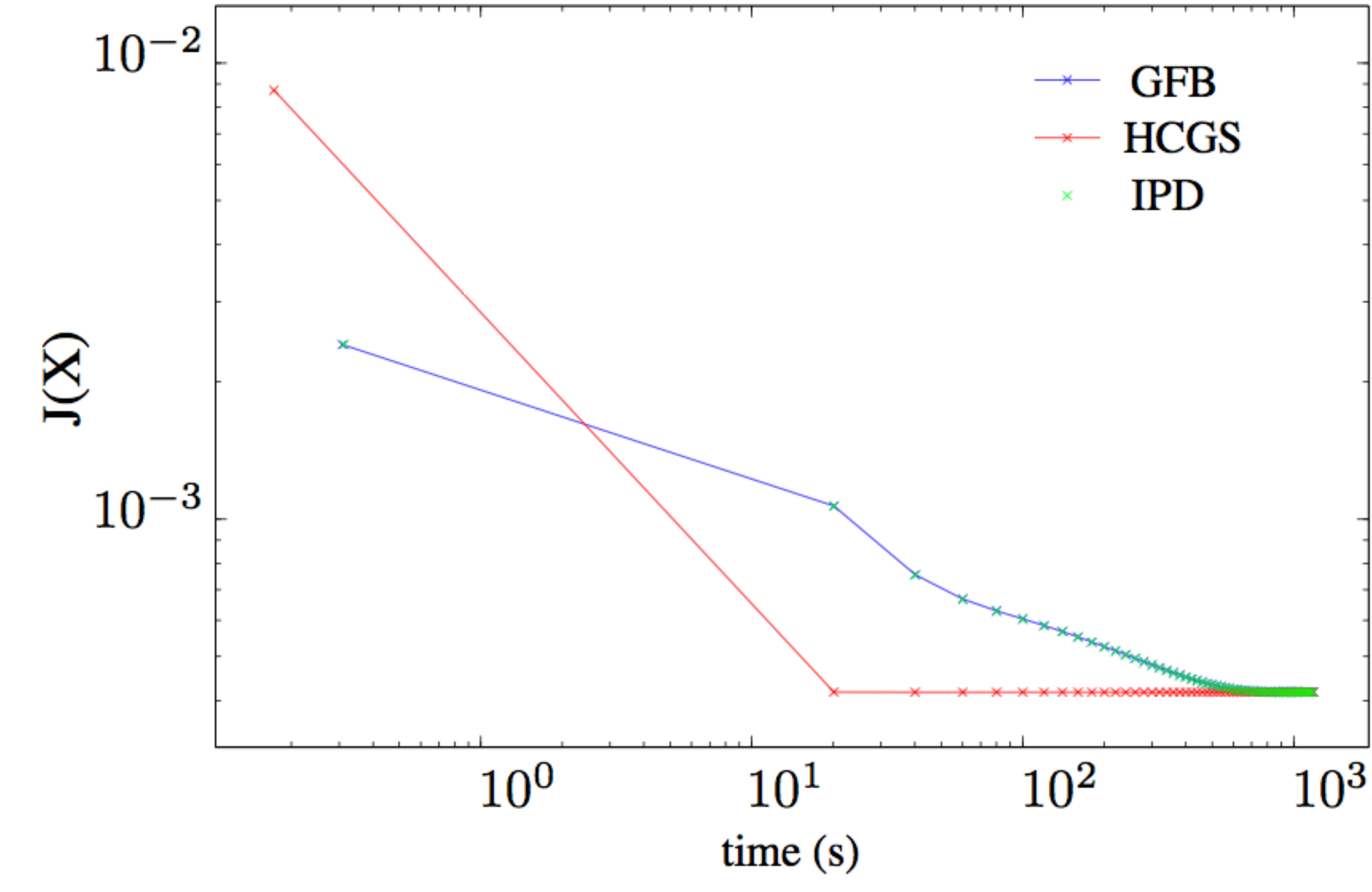}
		\begin{center}
			\vspace{-.5cm}
					~~~~~~~~~~	(b)   
		\end{center}
		\end{minipage} 
		\end{tabular}
	\caption{Comparison of objective values for $N=400$, (a) as a function of the iteration count, and (b) as a function of time.\label{fig:compfig2}}
\end{figure} 

\begin{figure}\begin{center} 
	\includegraphics[width=0.8\textwidth]{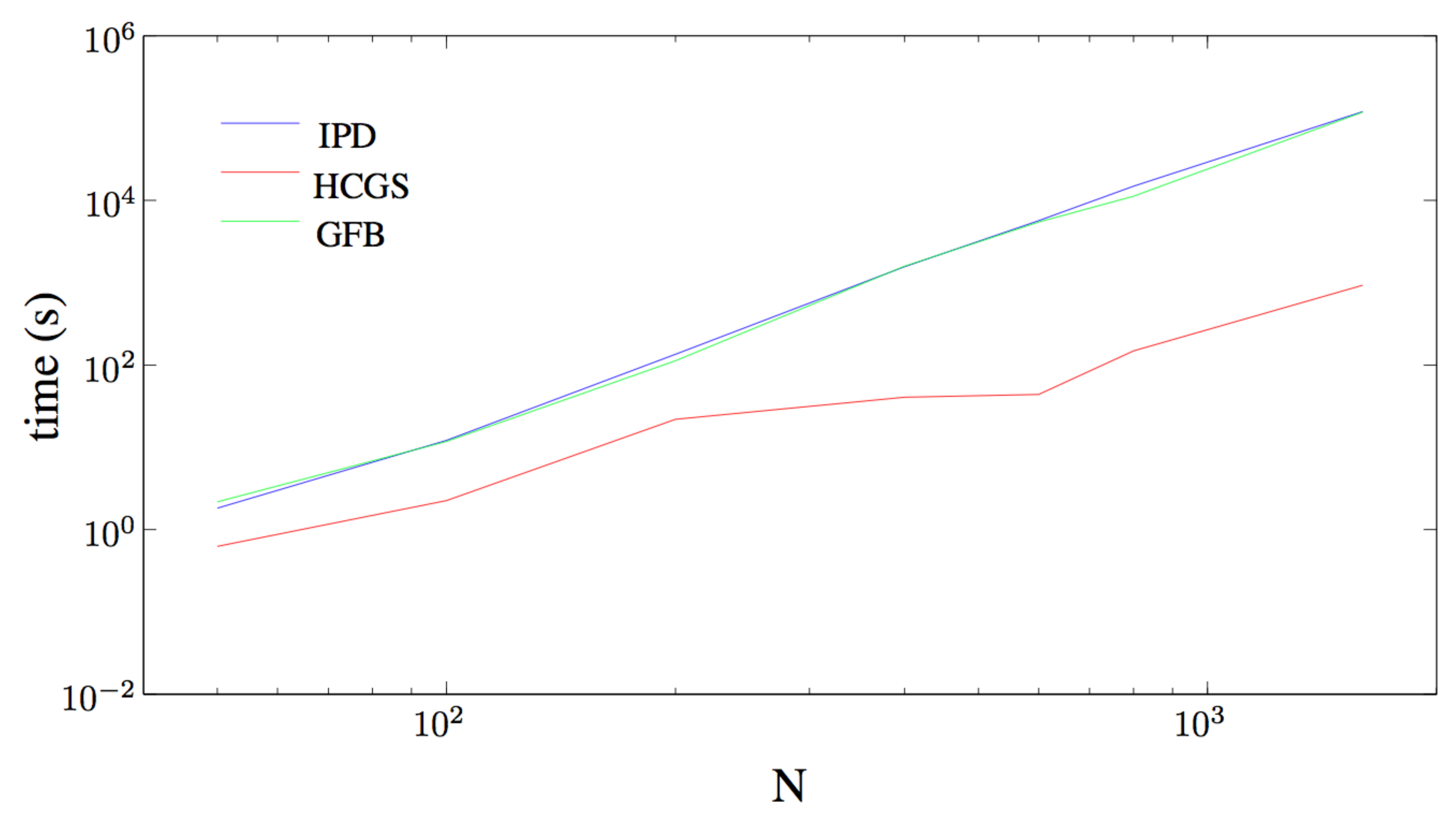}\end{center}
	\caption{Required time for convergence as a function of $N$ ($40\%$ observed entries), for the sparse - low rank experiment.\label{fig:time}}
\end{figure}

We considered the task of recovering a matrix from a subset of its
entries.  To this end in each simulation we generated two $N \times 5$
random matrices with entries drawn from the uniform distribution.
$90\%$ of the entries corresponding to a subset of uniformly
distributed indices were then set to zero.  The resulting matrices,
denoted by $U$ and $V$, were then used to obtain a sparse and low rank
matrix $UV^\top$. This matrix was corrupted by zero-mean Gaussian
noise with variance $\sigma^2=10^{-4}$ to obtain the observation
matrix $Y$. A fraction $f\in\{0.05,0.4\}$ of entries of $Y$
were used for recovery; see Figure \ref{fig:compfig} for an
illustration.

We compared HCGS with the two algorithms proposed in \cite{emile},
namely \emph{generalized forward-backward} (GFB) \cite{raguet} and \emph{incremental
  proximal descent} (IPD) \cite{bertsekas}. Both these algorithms solve a convex matrix
recovery problem that aims at finding a matrix that is simultaneously
low rank and sparse. This problem is:
\begin{equation}\label{form1}
	\min\left\{ J(X):=\frac{1}{2p}\Vert \Omega(Y-X)\Vert^2_F+\lambda_1\Vert X\Vert_1+\lambda_2\Vert X\Vert_{tr}~:~X\in\mathbb{R}^{N\times N}\right\}
\end{equation} 
where $\Vert \cdot\Vert_F$ is the Frobenius norm, $p$ is the number of
observed entries and $\Omega:\mathbb{R}^{N\times N}\rightarrow
\mathbb{R}^{N\times N}$ is the sampling operator defined entry-wise by
$\Omega(X)_{ij}=X_{ij}$ if the entry indexed by $(i,j)$ is observed,
$\Omega(X)_{ij}=0$ otherwise. In our experiments we set
$\lambda_1=\tfrac{1}{N^2},~\lambda_2=\tfrac{10^{-3}}{N^2}$ and used GFB and IPD to obtain
optimal estimates $\hat{X}_{\mathrm{GFB}}$ and
$\hat{X}_{\mathrm{IPD}}$, respectively. We then set $\tau:=\Vert
\hat{X}_{\mathrm{GFB}}\Vert_{tr}$ and used HCGS to solve the
constrained formulation equivalent to \eqref{form1}, namely:
\begin{equation}\label{form2}
	\min\left\{\frac{1}{2p}\Vert \Omega(Y-X)\Vert^2_F+\lambda_1\Vert X\Vert_1~:~
	\Vert X\Vert_{tr}\leq \tau
	,~X\in\mathbb{R}^{N\times N}\right\}~.
\end{equation}
The comparisons were performed on an Intel Xeon with 8 cores and 24GB of memory.


Figure \ref{fig:compfig2} shows the evolution of objective values for
$N = 400$.  Note that for the sake of comparison we have reported the
objective of the optimization problem \eqref{form1} even though HCGS
actually solves the equivalent problem in \eqref{form2}. The same
applies to the attained objective function value $J_{k^*}$
in Table \ref{comptab}. In this table we have also compared the
different algorithms in terms of CPU time\footnote{In Figure
  \ref{fig:compfig2}b the CPU times include also the evaluation of the
  objective value in \eqref{form1} which requires computing the
  singular values of the estimate. In the case of HCGS, this is
  required only for the sake of comparison with GFB and IPD. In
  contrast, for Table \ref{comptab} the objective value of
  \eqref{form1} is computed only upon termination.  }, relative change
in the objective function and number of iterations upon
termination. In all the cases we terminated the algorithms at
iteration $k^*$ when the relative change $r_{k^*}$ in the objective
value:
\begin{equation}\label{relchange}
r_{k^*}= \left\vert  \frac{f_{k^*}-f_{k^*-1}}{f_{k^*-1}} \right\vert
\end{equation}
was less that $10^{-7}$. Note that $f$ in \eqref{relchange} refers to
the objective function actually minimized by each algorithm; this is
not necessarily the objective function $J$ in \eqref{form1}.  
Figure \ref{fig:time} shows the time complexity as a function of
$N$. Finally, in Table \ref{tab:comptab2} we have reported the average
time per iteration as a function of $N$. 

From these figures and tables, we see that the running time of 
HCGS scales as $\calO(N^2)$ with the matrix size, whereas 
both GFB and IPD scale as $\calO(N^3)$.

\begin{table}
\caption{Comparison of different algorithms for convex matrix recovery.}
\vspace{.3cm}
\begin{center}
$5\%$ observed entries\\\vspace{.4cm}
{\scriptsize \begin{tabular}{c|c|c|c|c|c|cc}
N&&$J_{k^*}  (\times 10^{-4}$)&$r_{k^*}$($\times 10^{-8}$)&time
(s)&$k^*$ \\\hline\hline
\multirow{2}{*}{50} & HCGS & 6.77 & 8.35  & 2.38&1605\\
& GFB & 6.87 & 5.15 & 0.09 & 45 \\
& IPD & 6.76 & 9.07& 0.09 & 50 \\ \hline\hline
\multirow{2}{*}{100} & HCGS & 4.41 & 7.52 & 3.60 & 1462 \\
& GFB & 4.45& 9.67& 2.57& 385\\
& IPD & 4.40&9.80& 2.5 & 388\\ \hline\hline
\multirow{2}{*}{200} & HCGS & 4.16  & 2.42  & 66 & 3308 \\
& GFB & 4.17  & 9.95  & 113 & 4136\\
& IPD & 4.16 & 9.83  & 123 & 4645\\ \hline\hline
\multirow{2}{*}{400}& HCGS & 2.98  & 3.35 & 176  & 4555\\
& GFB & 2.98  & 9.99  & 2333 & 15241\\
& IPD & 2.98  & 9.99  & 2389  & 15665\\\hline \hline
\multirow{2}{*}{600} & HCGS & 2.46  & 7.15  & 158  & 2797\\
& GFB & 2.45 & 9.99  & 13157   & 36049\\
& IPD & 2.45  & 9.99  & 13080  & 36408\\\hline  \hline
\multirow{2}{*}{800} & HCGS & 2.20 & 5.47  & 478 & 5197\\
& GFB & 2.20 & 9.99  & 40779  & 61299\\
& IPD & 2.20  & 9.99  & 41263  & 61529
\\\hline  \hline
\end{tabular}\label{comptab}}\end{center}
\vspace{.3cm}
\begin{center}$40\%$ observed entries\\\scriptsize{
\vspace{.4cm}
  \begin{tabular}{c|c|c|c|c|c|cc}
N&&$J_{k^*}  (\times 10^{-4}$)&$r_{k^*}$($\times 10^{-8}$)&time
(s)&$k^*$ \\\hline\hline
\multirow{2}{*}{50} & HCGS & 18.51 & 8.85& 0.62& 427\\
& GFB & 18.66& 8.75 & 1.82& 751\\
& IPD & 18.52& 9.96& 2.17& 1004\\ \hline\hline
\multirow{2}{*}{100} & HCGS & 12.09& 4.96& 2.25& 835\\
& GFB & 12.15& 9.84& 12.1 & 1681\\
& IPD & 12.10 & 6.50&   11.8& 1697\\ \hline\hline
\multirow{2}{*}{200} & HCGS & 7.65 & 9.64 & 21  & 1410\\
& GFB & 7.66 & 9.98 & 134  & 3521\\
& IPD & 7.65& 7.66 & 112  & 3033\\ \hline\hline
\multirow{2}{*}{400}& HCGS & 4.39 & 2.54 & 40  & 1281\\
& GFB & 4.39 & 9.96 & 1559  & 7379\\
& IPD & 4.39 & 9.46 & 1580  & 7515\\\hline \hline
\multirow{2}{*}{600} & HCGS & 3.41 & 7.65 & 43  & 760\\
& GFB & 3.41 & 9.98 & 5664  & 10793\\
& IPD & 3.41 & 9.99 & 5446  & 10841\\\hline  \hline
\multirow{2}{*}{800} & HCGS & 2.68 & 1.37 & 148  & 1378\\
& GFB & 2.68 & 9.98 & 14897  & 15615\\
& IPD & 2.68 & 9.99 & 11242  & 15647
\\\hline  \hline
\multirow{2}{*}{1600} & HCGS & 1.62 & 9.91 & 929 & 2931\\
& GFB & 1.62 & 9.99 & 119724 & 36573\\
& IPD & 1.62 & 9.99 & 118223 & 36583\\\hline   \hline
\end{tabular}\label{tab:comptab}}\end{center}
\end{table}

\begin{table}[t]
	\caption{Average time (in seconds) per iteration as a function of $N$ ($40\%$ observed entries), for the sparse - low rank experiment.}
	\vspace{.3cm} 
	\begin{center}
{\small	\begin{tabular}{c|c|c|c|c|c|}
&\multicolumn{5}{|c|}{$N$}\\
&200&400&600&800&1600\\\hline\hline
HCGS&0.0155&0.032&0.058&0.107&0.317\\  \hline
GFB&  0.038&0.211&0.524&0.954&3.273\\  \hline
IPD&  0.037&0.210&0.502&0.718&3.231\\\hline\hline
	\end{tabular}\label{tab:comptab2}}\end{center}
\end{table}


\subsection{Sparse PCA}


The second set of simulations assesses the computational
efficiency of HCGS on the convex relaxation of sparse PCA \eqref{eq:sparse_pca}.
Similar to \cite{aspremont}, we generated random matrices $C$
as follows. For each size $n$, we drew an $n\times n$
matrix $U$ with uniformly distributed elements in $[0,1]$. 
Then we generated a vector $v\in\R^n$ from the uniform distribution
and set a random 90\% of its components to zero. We then set 
$$
C=UU\trans + 10vv\trans.
$$ 
We solved \eqref{eq:sparse_pca}
for $\lambda=1$ with HCGS (Algorithm \ref{alg:sparse_pca}) 
and the Nesterov smoothing method of \cite{aspremont} which optimizes
the dual problem of \eqref{eq:sparse_pca}.
We implemented both algorithms in Matlab and used a cluster
with 24 cores and sufficient memory. For HCGS we used the power method
with a tolerance of $10^{-6}$, for computing dominant eigenvectors.
We also rescaled the objective function by $n$
in order to keep $L_g$ small enough (see Section \ref{sec:implementation}).
For Nesterov smoothing we used $\mu = \tfrac{10^{-6}}{2\log n}$.

 In Figure \ref{fig:spca}, we plot the computational times required to 
 attain relative change of $10^{-5}$ in the objective. 
 We note that the objective functions are different, since 
HCGS optimizes \eqref{eq:sparse_pca} whereas the method of
 \cite{aspremont} optimizes the dual problem.
In fact, we have verified that the duality gap estimates are consistently larger for the latter 
and hence the running times for Nesterov smoothing are optimistic.
We observe that the running time scales roughly as $\calO(n^2)$ for HCGS
whereas Nesterov smoothing scales worse than $\calO(n^3)$.



 \begin{figure}[t] 
 \begin{center}
\includegraphics[width=.8\textwidth]{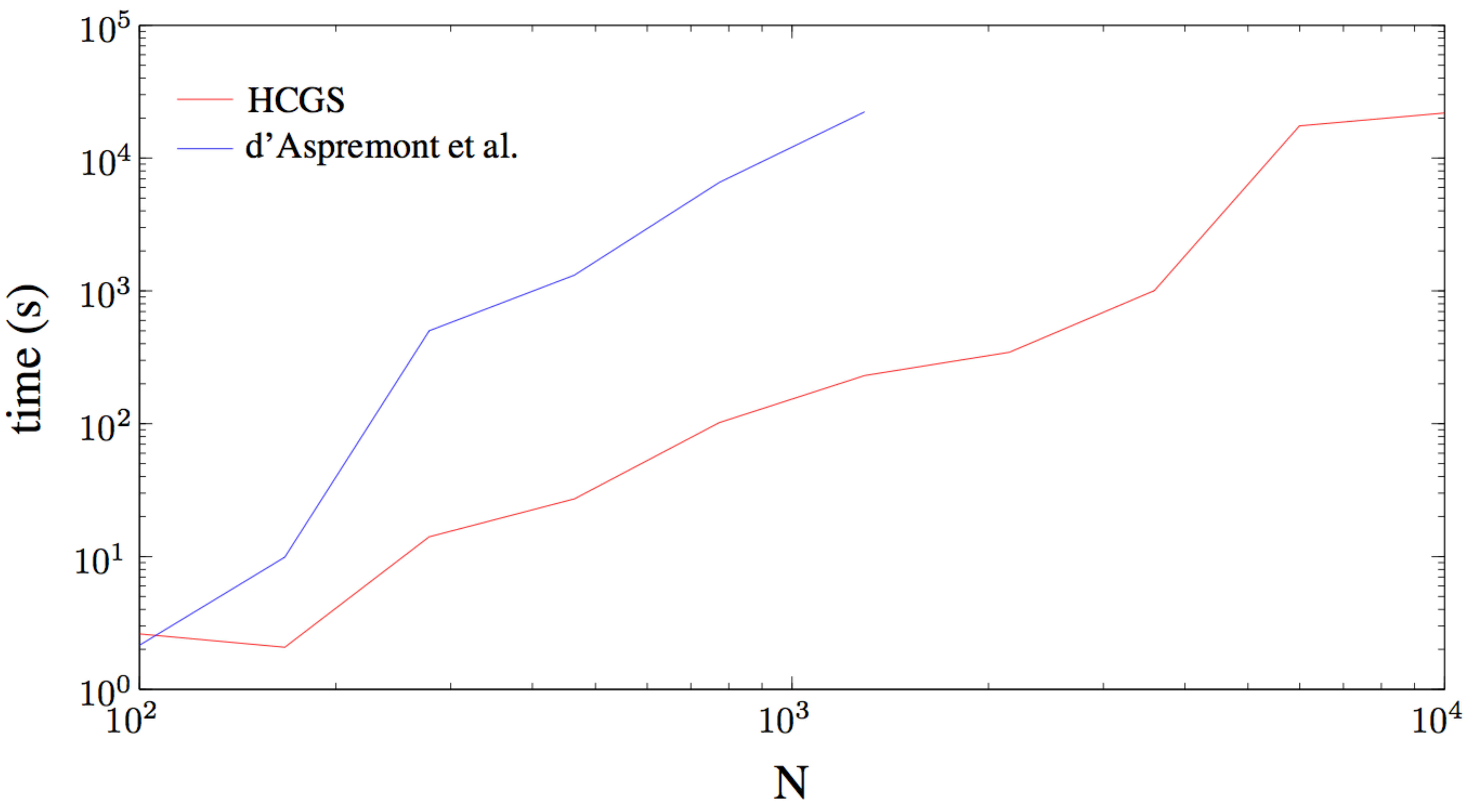}
\caption{Computational time (in seconds) versus matrix size for the sparse PCA experiment.}
 \end{center} 
\label{fig:spca}
 \end{figure}


\section{Conclusion}

We have studied the hybrid conditional gradient - smoothing algorithm (HCGS)
for solving composite convex optimization problems which contain several
terms over a bounded set. Examples of these include regularization problems
with several norms as penalties and a norm constraint.
HCGS extends conditional gradient methods to cases with multiple nonsmooth
terms, in which standard conditional gradient methods may be difficult to apply.
The HCGS algorithm borrows techniques from smoothing proximal methods
and requires first-order computations (subgradients and proximity operations).
Moreover, it exhibits convergence in terms of the objective values at an
$\calO\left(\frac{1}{\varepsilon^2}\right)$ rate of iterations.  
Unlike proximal methods, HCGS benefits from the advantages of conditional
gradient methods, which render it more efficient on certain large scale optimization problems.
We have demonstrated these advantages with simulations on two matrix optimization problems: 
regularization of matrices with combined $\ell_1$ and trace norm penalties;
and a convex relaxation of sparse PCA.

\section*{Acknowledgments}

The research leading to the above results has received funding from the European Union Seventh Framework
Programme (FP7 2007-2013) under grant agreement No. 246556.
The research has also received funding from the
European Union Seventh Framework Programme (FP7 2007-2013) under
grant agreement No. 246556. Research supported also by: Research Council
KUL: GOA/10/09 MaNet, PFV/10/002 (OPTEC), several PhD/postdoc and
fellow grants; Flemish Government:IOF: IOF/KP/SCORES4CHEM, FWO:
PhD/postdoc grants, projects: G.0377.12 (Structured systems), G.083014N
(Block term decompositions), G.088114N (Tensor based data similarity),
IWT: PhD Grants, projects: SBO POM, EUROSTARS SMART, iMinds 2013,
Belgian Federal Science Policy Oce: IUAP P7/19 (DYSCO, Dynamical 
systems, control and optimization, 2012-2017), EU: FP7-SADCO (MC ITN-
264735), ERC AdG A-DATADRIVE-B (290923), COST: Action ICO806: IntelliCIS.


\bibliographystyle{plain}

\end{document}